\documentclass[11pt]{article}%
\usepackage{amsfonts}
\usepackage{amsmath}
\usepackage{amssymb}
\usepackage{graphicx}%
\providecommand{\U}[1]{\protect\rule{.1in}{.1in}}
\newtheorem{theorem}{Theorem}

\newtheorem{corollary}[theorem]{Corollary}

\newtheorem{lemma}[theorem]{Lemma}

\newtheorem{proposition}[theorem]{Proposition}
\newtheorem{remark}[theorem]{Remark}

\newenvironment{proof}[1][Proof]{\noindent\textbf{#1.} }{\ \rule{0.5em}{0.5em}}
{\catcode`\@=11\global\let\AddToReset=\@addtoreset
\AddToReset{equation}{section}

\AddToReset{theorem}{section}

\begin{document}

\title{Large solutions of elliptic systems of second order and applications to the
biharmonic equation\thanks{{\scriptsize The first author was supported by
Fondecyt 70100002 and Ecos-Conicyt C08E04, and the second and third author
were supported by Fondecyt 1070125, Fondecyt 1070951 as well as by
Ecos-Conicyt C08E04.}}}
\author{Marie-Fran\c{c}oise BIDAUT-VERON\thanks{Laboratoire de Math\'{e}matiques et
Physique Th\'{e}orique, CNRS UMR 6083, Facult\'{e} des SCiences, 37200 Tours
France. E-mail address: veronmf@univ-tours.fr}
\and Marta GARC\'{I}A-HUIDOBRO\thanks{Departamento de Matem\'{a}ticas, Pontificia
Universidad Cat\'{o}lica de Chile, Casilla 306, Correo 22, Santiago de Chile.
E-mail address: mgarcia@mat.puc.cl}
\and Cecilia YARUR\thanks{Departamento de Matem\'{a}tica y C.C., Universidad de
Santiago de Chile, Casilla 307, Correo 2, Santiago de Chile. E-mail address:
cecilia.yarur@usach.cl}}
\maketitle

\begin{abstract}
In this work we study the nonnegative solutions of the elliptic system
\[
\Delta u=|x|^{a}v^{\delta},\qquad\Delta v=|x|^{b}u^{\mu}%
\]
in the superlinear case $\mu\delta>1,$ which blow up near the boundary of a
domain of $\mathbb{R}^{N},$ or at one isolated point. In the radial case we
give the precise behavior of the large solutions near the boundary in any
dimension $N$. We also show the existence of infinitely many solutions blowing
up at $0.$ Furthermore, we show that there exists a global positive solution
in $\mathbb{R}^{N}\backslash\left\{  0\right\}  ,$ large at $0,$ and we
describe its behavior. We apply the results to the sign changing solutions of
the biharmonic equation%
\[
\Delta^{2}u=\left\vert x\right\vert ^{b}\left\vert u\right\vert ^{\mu}.
\]
Our results are based on a new dynamical approach of the radial system by
means of a quadratic system of order 4, introduced in \cite{BGi}, combined
with the nonradial upper estimates of \cite{BGr}.

\end{abstract}

\begin{center}

\end{center}

\noindent{\footnotesize 2000}\textit{{\footnotesize Mathematics Subject
Classification: 35J60,35B40,35C20,34A34.}}\newline\textit{{\footnotesize Key
words: Semilinear elliptic systems, Boundary blow-up, Keller-Osserman
estimates, Asymptotic behavior, Biharmonic equation.}}

\section{ Introduction\label{intro}}

This article is concerned with the nonnegative large solutions of the elliptic
system
\begin{equation}
\left\{
\begin{array}
[c]{l}%
\Delta u=|x|^{a}v^{\delta}\\
\Delta v=|x|^{b}u^{\mu},
\end{array}
\right.  \label{sab}%
\end{equation}
in two cases: solutions in a bounded domain $\Omega$ in $\mathbb{R}^{N}$,
which blow up at the boundary, that is
\begin{equation}
\lim_{d(x,\partial\Omega)\rightarrow0}u(x)=\lim_{d(x,\partial\Omega
)\rightarrow0}v(x)=\infty, \label{vig}%
\end{equation}
where $d(x,\partial\Omega)$ is the distance from $x$ to $\partial\Omega$; or
solutions in $\Omega\backslash\left\{  0\right\}  $ which blow up at $0:$
\begin{equation}
\lim_{x\rightarrow0}u(x)=\infty\quad\text{or}\quad\lim_{x\rightarrow
0}v(x)=\infty. \label{vig0}%
\end{equation}
We study the superlinear case, where $\mu,\delta>0$, and
\begin{equation}
D=\mu\delta-1>0, \label{dd}%
\end{equation}
and $a,b$ are real numbers such that
\begin{equation}
a,b>\max\{-2,-N\}. \label{abn}%
\end{equation}

First we recall some well-known results in the scalar case of the Emden-Fowler
equation
\begin{equation}
\Delta U=U^{Q} \label{EF}%
\end{equation}
with $Q>1.$ Concerning the boundary blow-up problem, there exists a unique
solution $U$ in $\Omega$ such that $\lim_{d(x,\partial\Omega)\rightarrow
0}U(x)=\infty,$ and near $\partial\Omega$
\[
U(x)=Cd(x,\partial\Omega)^{-2/(Q-1)}(1+o(1)),
\]
where $C=C(Q)$. Several researchs on the more general equation
\[
\Delta U=p(x)f(U)
\]
have been done with different assumptions on $f$ and on the weight $p$, with
asymptotic expansions near $\partial\Omega$ , see for instance \cite{bm1},
\cite{bm2}, \cite{CoDu}, \cite{delp-l}, \cite{lm2}, \cite{LoNi}, \cite{marv1},
\cite{marv2}, \cite{veron}; see also \cite{be}, \cite{di-l} for quasilinear
equations. These results rely essentially on the \textit{comparison principle}
valid for this equation, and the construction of supersolutions and
subsolutions. \medskip

The existence and the behavior of solutions of (\ref{EF}) in $\Omega
\backslash\left\{  0\right\}  $ which blow up at 0:
\[
\lim_{x\rightarrow0}U(x)=\infty,
\]
called large (or singular) at 0, have also been widely investigated during the
last decades, see for example \cite{veron2}, and the references therein. There
exists a particular solution in $\mathbb{R}^{N}\backslash\left\{  0\right\}  $
whenever $Q<N/(N-2)$ or $N=1,2,$ given by $U^{\ast}(x)=C^{\ast}\left\vert
x\right\vert ^{-2/(Q-1)},$ with $C^{\ast}=C^{\ast}(Q,N).$ If $Q\geq N/(N-2)$,
there is no large solution at $0$, and the singularity is removable. If
$Q<N/(N-2)$ or $N=2$, any large solution satisfies $\lim_{\left\vert
x\right\vert \rightarrow0}\left\vert x\right\vert ^{2/(Q-1)}U=C^{\ast}$, or
\begin{equation}
\lim_{\left\vert x\right\vert \rightarrow0}\left\vert x\right\vert
^{N-2}U=\alpha>0\text{\quad if }N>2,\qquad\lim_{\left\vert x\right\vert
\rightarrow0}\left\vert \ln\left\vert x\right\vert \right\vert U=\alpha
>0,\text{\quad if }N=2. \label{cdf}%
\end{equation}
There exist solutions of each type, distinct from $U^{\ast}.$ Moreover, up to
a scaling, there exists a unique positive radial solution in $\mathbb{R}%
^{N}\backslash\left\{  0\right\}  $, such that (\ref{cdf}) holds and
$\lim_{\left\vert x\right\vert \rightarrow\infty}\left\vert x\right\vert
^{2/(Q-1)}U=C^{\ast}$, see \cite{veron2} and also \cite{BGi}.\medskip\medskip\ 

In \textbf{Section \ref{boun}} we consider the blow up problem of system
(\ref{sab}) at the boundary. \medskip

Up to our knowledge all the known results for systems are related with systems
for which some comparison properties hold, for example
\[
\left\{
\begin{array}
[c]{l}%
\Delta u=u^{s}v^{\delta},\\
\Delta v=u^{\mu}v^{m},
\end{array}
\right.
\]
where $s,m>1$, $\delta,\mu>0$, and $\delta\mu\leq(s-1)(m-1),$ of competitive
type, see \cite{gmr1}, or $\delta,\mu<0$, of cooperative type, see
\cite{DDGM}; see also some extensions to problems with weights in \cite{MHTL},
or with quasilinear operators in \cite{gm2}, \cite{Wa}, \cite{WuYa}, and
cooperative systems of Lotka-Volterra in \cite{gms}.\medskip

On the contrary the problem (\ref{sab})-(\ref{vig}) has been the object of
very few works, because it brings many difficulties. The main one is the
\textit{lack of a comparison principle} for the system. As a consequence all
the methods of supersolutions, subsolutions and comparison, valid for the case
of a single equation fail.\medskip

Until now the existence of large solutions is an open question in the
nonradial case. In the radial case the problem was studied in \cite{gmls2},
without weights: $a=b=0$. It was shown that there are infinitely many
nonnegative radial solutions to (\ref{sab}) which blow up at the boundary of a
ball provided that (\ref{dd}) holds, and no blow up occurs otherwise. In
particular, there exist solutions even in the case where either $u$ or $v$
vanishes at $0$. This shows \textit{the lack of a Harnack inequality, }even in
the radial case. The precise behavior of the solutions was obtained in
\cite{gmls2} for  $N=1,$ $a=b=0$, where system (\ref{sab}) is autonomous,
with an elaborate proof wich  could not be extended to higher dimension$.$
\medskip

Our first main result \textit{solves this question in any dimension}, with
possible weights, and moreover we give an expansion of order 1 of the solutions:

\begin{theorem}
\label{main}Let $(u,v)$ be any radial nonnegative solution of (\ref{sab})
defined for $r\in$ $\left(  r_{0},R\right)  $, $r_{0}\geq0$, \textit{unbounded
at }$r=R$\textit{. Then } $\lim\limits_{r\rightarrow R}u(r)=\lim
\limits_{r\rightarrow R}v(r)=\infty$, and $u,v$ admit the following expansions
near $R:$
\begin{equation}
u(r)=A_{1}d(r)^{-\gamma}(1+O(d(r))),\qquad v(r)=B_{1}d(r)^{-\xi}%
(1+O(d(r))),\text{ } \label{esti}%
\end{equation}
where $d(r)=R-r$ is the distance to the boundary, and
\begin{equation}
\gamma=\frac{2(1+\delta)}{D},\qquad\xi=\frac{2(1+\mu)}{D}, \label{gks}%
\end{equation}%
\begin{equation}
A_{1}=(\gamma(\gamma+1)(\xi(\xi+1))^{\delta})^{1/D},\qquad B_{1}=(\xi
(\xi+1)(\gamma(\gamma+1))^{\mu})^{1/D}. \label{A1B1}%
\end{equation}

\end{theorem}

Our proof is essentially based on a new dynamical approach of system
(\ref{sab}), initiated in \cite{BGi}: we reduce the problem to a quadratic, in
general nonautonomous, system of order $4$, which, under the assumptions of
Theorem \ref{main}, can be reduced to a nonautonomous perturbation of a
quadratic system of order $2$. We then show the convergence of the solution of
the original system to a suitable fixed point by using the perturbation
arguments of \cite{lr}. \medskip

Theorem \ref{main} can be applied to \textit{sign changing} solutions of some
elliptic systems, in particular to the biharmonic equation, where $\delta=1$:

\begin{corollary}
\label{bila}Let $\mu>1$, $b\in\mathbb{R}$. Then any radial solution $u$ of the
problem
\begin{equation}
\Delta^{2}u=\left\vert x\right\vert ^{b}\left\vert u\right\vert ^{\mu}%
\quad\text{in}\quad(r_{0},R),\qquad u(R)=\infty, \label{bih}%
\end{equation}
satisfies
\begin{equation}
u(r)=Ad(r)^{-4/(\mu-1)}(1+O(d(r))), \label{ddd}%
\end{equation}
with $A^{\mu-1}=8(\mu+3)(\mu+1)(3\mu-1)(\mu-1)^{-4}$.
\end{corollary}

We notice here a case where we find an \textit{explicit} solution: for $N>4$
and $\mu=\frac{N+4}{N-4}$, equation $\Delta^{2}u=u^{\mu}$ admits the solution
in the ball $B(0,1),$
\[
u(r)=C(1-r^{2})^{(4-N)/2},\qquad C^{8/(N-4)}=N(N-4)(N^{2}-4),
\]
and $v=\Delta u=C(N-4)(1-r^{2})^{-N/2}(N-2r^{2})\geq0$, and (\ref{esti}) and
(\ref{ddd}) hold with $\gamma=\frac{N-4}{2},\xi=\frac{N}{2}$.\medskip\medskip

In \textbf{Section \ref{orig}} we consider the problem of large solutions at
the origin, that is (\ref{sab})-(\ref{vig0}). \medskip

System (\ref{sab}) admits a particular radial positive solution $(u^{\ast
},v^{\ast})$, given by
\begin{equation}
u^{\ast}(r)=A_{N}r^{-\gamma_{a,b}},\qquad v^{\ast}(r)=B_{N}r^{-\xi_{a,b}%
},\quad r=|x|, \label{zero}%
\end{equation}
where
\begin{equation}
\gamma_{a,b}=\frac{(2+a)+(2+b)\delta}{D}>0,\qquad\xi_{a,b}=\frac
{(2+b)+(2+a)\mu}{D}>0, \label{ga}%
\end{equation}%
\[
A_{N}^{D}=\gamma_{a,b}(\gamma_{a,b}-N+2)\left(  \xi_{a,b}(\xi_{a,b}%
-N+2)\right)  ^{\delta},\quad B_{N}^{D}=\xi_{a,b}(\xi_{a,b}-N+2)\left(
\gamma_{a,b}(\gamma_{a,b}-N+2)\right)  ^{\mu},
\]
whenever
\begin{equation}
\min\left\{  \gamma_{a,b},\xi_{a,b}\right\}  >N-2,\quad\text{or }N=1,2.
\label{cde}%
\end{equation}
Note that in particular $\gamma_{0,0}=\gamma,\xi_{0,0}=\xi$.\medskip

The problem has been initiated in \cite{Y1}\textbf{ }and \cite{BGr}, see also
\cite{Y2}. Let us recall an important result of \cite{BGr} giving upper
estimates for system (\ref{sab}) in the \textit{nonradial} case, stated for
$N\geq3$, but its proof is valid for any $N\geq1$. It is not based on
supersolutions, but on estimates of the mean value of $u,v$ on
spheres:\medskip

\textit{Keller-Osserman type estimates} \cite{BGr}. \textit{\ Let} $\Omega$
\textit{be a domain of }$\mathbb{R}^{N}(N\geq1)$\textit{, containing}
$0$\textit{,} \textit{and} $u,v\in C^{2}(\Omega\backslash\left\{  0\right\}
)$ \textit{be any nonnegative subsolutions of (\ref{sab}), that is,}
\[
\left\{
\begin{array}
[c]{l}%
-\Delta u+|x|^{a}v^{\delta}\leq0,\\
-\Delta v+|x|^{b}u^{\mu}\leq0,
\end{array}
\right.
\]
\textit{\ with }$\mu,\delta$\textit{\ satisfying (\ref{dd}). Then there exists
}$C=C(a,b,\delta,\mu,N)$\textit{\ such that near }$x=0,$\textit{\ }
\begin{equation}
u(x)\leq C\left\vert x\right\vert ^{-\gamma_{a,b}},\qquad v(x)\leq C\left\vert
x\right\vert ^{-\xi_{a,b}}.\text{ } \label{oe}%
\end{equation}
Moreover, one finds in \cite{BGr} a quite exhaustive study about all the
\textit{possible} behaviors of the solutions (radial or not) in $\Omega
\backslash\left\{  0\right\}  $.\medskip

Here we complete those results by proving the existence of local radial
solutions large at $0$ of each of the types described in \cite{BGr}, see
Propositions \ref{ex1}, \ref{mani} in Section \ref{orig}. By using these
results, we obtain our second main result in this work, which is the following
global existence theorem:

\begin{theorem}
\label{global}Assume that $N\geq2$ and that (\ref{cde}) holds. Then there
exists a radial positive global solution of system (\ref{sab}) in
$\mathbb{R}^{N}\backslash\left\{  0\right\}  $, large near 0, unique up to a
scaling, such that
\begin{equation}
\lim_{r\rightarrow\infty}r^{\gamma_{a,b}}u=A_{N},\qquad\lim_{r\rightarrow
\infty}r^{\xi_{a,b}}v=B_{N}; \label{fur}%
\end{equation}
and, for $N>2$, and up to a change of $u,\mu,a$, into $v,\delta,b$, when
$\delta<\frac{N+a}{N-2}$, it satisfies
\[
\lim\limits_{r\rightarrow0}r^{N-2}u=\alpha>0,\qquad\left\{
\begin{array}
[c]{rcl}%
\lim\limits_{r\rightarrow0}r^{N-2}v & = & \beta>0,\qquad\text{if }\mu
<\frac{N+b}{N-2},\\
\lim\limits_{r\rightarrow0}r^{(N-2)\mu-(2+b)}v & = & \beta>0,\qquad\text{if
}\mu>\frac{N+b}{N-2},\\
\lim\limits_{r\rightarrow0}r^{N-2}\left\vert \ln r\right\vert ^{-1}v & = &
\beta>0,\qquad\text{if }\mu=\frac{N+b}{N-2},
\end{array}
\right.
\]
and for $N=2$,
\[
\lim_{r\rightarrow0}\left\vert \ln r\right\vert ^{-1}u=\alpha>0,\qquad
\lim_{r\rightarrow0}\left\vert \ln r\right\vert ^{-1}v=\beta>0.
\]

\end{theorem}

Our proof also relies on the dynamical approach of system (\ref{sab}) in
dimension $N$ by a quadratic autonomous system of order 4, given in
\cite{BGi}. Finally we give an application to the biharmonic equation:

\begin{corollary}
Let $N>2$. Assume that $1<\mu<\frac{N+2+b}{N-2}$. There exists a positive
global solution, unique up to a scaling, of equation
\[
\Delta^{2}u=|x|^{b}u^{\mu}%
\]
in $\mathbb{R}^{N}\backslash\left\{  0\right\}  $ , such that
\[
\lim_{r\rightarrow0}r^{N-2}u=\alpha>0,\qquad\lim_{r\rightarrow\infty
}r^{(4+b)/(\mu-1)}u=C,
\]
where $C^{\mu-1}=(4+b)(N+2+b-(N-2)\mu)\left(  2\mu+2+b)(N+b-(N-4)\mu\right)
(\mu-1)^{-4}$.
\end{corollary}

\section{Large solutions at the boundary\label{boun}}

This section is devoted to the study of the boundary blow up problem for
nonnegative radial solutions of (\ref{sab}). We begin by observing that system
(\ref{sab}) admits a scaling invariance: if $(u,v)$ is a solution, then for
any $\theta>0$,
\begin{equation}
r\mapsto(\theta^{\gamma_{a,b}}u(\theta r),\theta^{\xi_{a,b}}v(\theta r)),
\label{scaling}%
\end{equation}
where $\gamma_{a,b},\xi_{a,b}$ are defined in (\ref{ga}), is also a
solution.\medskip

\subsection{Existence and estimates of large solutions}

We say that a nonnegative solution $(u,v)$ of (\ref{sab}) defined in $(0,R)$
is \textit{regular} at $0$ if $u,v\in C^{2}\left(  0,R\right)  \cap C(\left[
0,R)\right)  $. Then $u,v\in C^{1}(\left[  0,R)\right)  $ when $a,b\geq-1$,
and moreover $u^{\prime}(0)=v^{\prime}(0)=0$ when $a,b>-1$, and $u,v\in
C^{2}(\left[  0,R)\right)  $ when $a,b\geq0$. \medskip

We first give an existence and uniqueness result for regular solutions:

\begin{proposition}
\label{ereg} Assume (\ref{abn}) and only that $D=\delta\mu-1\neq0$. Then for
any $u_{0},\ v_{0}\geq0,$ there exists a unique local regular solution $(u,v)
$ with initial data $(u_{0},v_{0})$.
\end{proposition}

The result follows from classical fixed point theorem when $u_{0},v_{0}>0$, by
writing the problem in an integral form:
\[
u(r)=u_{0}+\int_{0}^{r}\tau^{1-N}\int_{0}^{\tau}\theta^{N-1+a}v^{\delta
}(\theta)d\theta,\qquad v(r)=v_{0}+\int_{0}^{r}\tau^{1-N}\int_{0}^{\tau}%
\theta^{N-1+b}u^{\mu}(\theta)d\theta.
\]
In the case $u_{0}>0=v_{0}$, the existence can be obtained from the Schauder
fixed point theorem, and the uniqueness by using monotonicity arguments as in
\cite{gmls2}. We give an alternative proof in Section \ref{orig}, using the
dynamical system approach introduced in \cite{BGi}, which can be extended to
more general operators.\medskip

Next we show that all the nontrivial regular solutions blow up at some finite
$R>0$, and give the first upper estimates for any large solution. Our proofs
are a direct consequence of estimates (\ref{oe}).

\begin{proposition}
\label{ereg1}(i) Assume (\ref{dd}) and (\ref{abn}). For any regular
nonnegative solution $(u,v)\not \equiv (0,0),$ there exists $R$ such that $u$
and $v$ are unbounded near $R$. \medskip

(ii) Any solution $(u,v)$ which is nonnegative in an interval $(r_{0},R)$ and
unbounded at $R$, satisfies
\begin{equation}
\lim_{r\rightarrow R}u=\lim_{r\rightarrow R}v=\lim_{r\rightarrow R}u^{\prime
}=\lim_{r\rightarrow R}v^{\prime}=\infty. \label{cho}%
\end{equation}
and there exists $C=C(N,\delta,\mu)>0$ such that near $r=R$,
\begin{equation}
u(r)\leq C(R-r)^{-\gamma},\qquad v(r)\leq C(R-r)^{-\xi}. \label{gna}%
\end{equation}

\end{proposition}

\begin{proof}
(i) Let $(u,v)$ be any nontrivial regular solution. Suppose first that
$v_{0}>0$. Then from (\ref{sab}), $r^{N-1}u^{\prime}$ is positive for small
$r$, and nondecreasing, hence $u$ is increasing. If the solution is entire,
then it satisfies (\ref{oe}) near $\infty$: indeed by the Kelvin transform,
the functions
\[
\overline{u}(x)=\left\vert x\right\vert ^{2-N}u(x/\left\vert x\right\vert
^{2}),\qquad\overline{v}(x)=\left\vert x\right\vert ^{2-N}v(x/\left\vert
x\right\vert ^{2}),
\]
satisfy in $\ B(0,1)\backslash\left\{  0\right\}  $ the system
\[
\left\{
\begin{array}
[c]{c}%
-\Delta\overline{u}+\left\vert x\right\vert ^{\overline{a}}\overline
{v}^{\delta}=0,\\
-\Delta\overline{v}+\left\vert x\right\vert ^{\overline{b}}\overline{u}^{\mu
}=0,
\end{array}
\right.
\]
where $\overline{a}=(N-2)\delta-(N+2+a),\overline{b}=(N-2)\mu-(N+2+b)$, and
$\gamma_{a,b},\xi_{a,b}$ are replaced by $N-2-\gamma_{a,b},N-2-\xi_{a,b}$.
Then the estimate (\ref{oe}) for ($\overline{u},\overline{v})$ implies the one
for $(u,v)$ and thus $u$ tends to $0$ at $\infty$, which is contradictory.
Furthermore, from
\[
u\leq u_{0}+\frac{r^{2+a}}{(2+a)(N+a)}v^{\delta},\quad v\leq v_{0}%
+\frac{r^{2+b}}{(2+b)(N+b)}u^{\mu},
\]
$u$ and $v$ blow up at the same point $R>0$.\medskip

(ii) Since \textit{\ }$r^{N-1}u^{\prime}$ is increasing, it has a limit as
$r\rightarrow R$. If this limit is finite, then $u^{\prime}$ is bounded,
implying that $u$ has a finite limit; this contradicts our assumption. Thus
(\ref{cho}) holds. By (\ref{scaling}) we can assume $R=1$ and make the
transformation
\begin{equation}
r=\Psi(s)=\left\{
\begin{array}
[c]{ccc}%
(1+(N-2)s)^{-1/(N-2)}, &  & \text{if }N\neq2,\\
e^{-s}, &  & \text{if }N=2,
\end{array}
\right.  \label{tra}%
\end{equation}
(in particular $r=1-s$ if $N=1)$, so that $s$ describes an interval $\left(
0,s_{0}\right]  $, $s_{0}>0$, and we get the system
\begin{equation}
\left\{
\begin{array}
[c]{c}%
u_{ss}=F(s)v^{\delta}\\
v_{ss}=G(s)u^{\mu}%
\end{array}
\right.  \label{GH}%
\end{equation}
with
\begin{equation}
F(s)=r^{2N-2+a},\qquad G(s)=r^{2N-2+b}; \label{ffgg}%
\end{equation}
hence $\lim_{s\rightarrow0}F=\lim_{s\rightarrow0}G=1$. Then
\[
\left\{
\begin{array}
[c]{c}%
-u_{ss}+\frac{1}{2}v^{\delta}\leq0\\
-v_{ss}+\frac{1}{2}u^{\mu}\leq0
\end{array}
\right.
\]
in some interval $\left(  0,s_{1}\right]  $, thus from the Keller-Osserman
estimates (\ref{oe}), there exists $C=C(N,\delta,\mu)>0$ such that $u(s)\leq
Cs^{-\gamma},v(s)\leq Cs^{-\xi},$ near $s=0$ and (\ref{gna}) follows.
\end{proof}

\subsection{The precise behavior near the boundary}

In this section we prove Theorem \ref{main}.

\subsubsection{Scheme of the proof}

Consider a solution blowing up at $R=1$. In the case of dimension $N=1$, and
$a=b=0$, we have that $F\equiv G\equiv1$ in (\ref{ffgg}), and we are concerned
with the system
\begin{equation}
\left\{
\begin{array}
[c]{c}%
u_{ss}=v^{\delta}\\
v_{ss}=u^{\mu}.
\end{array}
\right.  \label{E}%
\end{equation}
Following the ideas of \cite{BGi}, we are led to make the substitution
\[
X(t)=-\frac{su_{s}}{u},\qquad Y(t)=-\frac{sv_{s}}{v},\qquad Z(t)=\frac
{sv^{\delta}}{u_{s}},\qquad W(t)=\frac{su^{\mu}}{v_{s}},
\]
where $t=\ln s$, $t$ describes $\left(  -\infty,t_{0}\right]  $, and we obtain
the autonomous system
\begin{equation}
\left\{
\begin{array}
[c]{rcl}%
X_{t} & = & X\left[  X+1+Z\right]  ,\\
Y_{t} & = & Y\left[  Y+1+W\right]  ,\\
Z_{t} & = & Z\left[  1-\delta Y-Z\right]  ,\\
W_{t} & = & W\left[  1-\mu X-W\right]  .
\end{array}
\right.  \label{M1}%
\end{equation}
We study the solutions in the region where $X,Y\geq0$ and $Z,W\leq0$. In this
region system (\ref{M1}) admits two fixed points
\begin{equation}
O=(0,0,0,0),\quad M_{0,1}=(\gamma,\xi,-1-\gamma,-1-\xi) \label{poi}%
\end{equation}
where $\gamma$ and $\xi$ are defined in (\ref{gks}). We intend to show that
trajectories associated to the large solutions converge to $M_{0,1}$. Observe
that system (\ref{E}) has a first integral, which is a \textit{crucial point}
in what follows:
\[
u_{s}v_{s}-\frac{u^{\mu+1}}{\mu+1}-\frac{v^{\delta+1}}{\delta+1}=C,
\]
equivalently
\[
e^{-2t}uv(XY+\frac{XZ}{\delta+1}+\frac{YW}{\mu+1})=C.
\]
Since any large solution at $r=1$ satisfies $\lim_{r\rightarrow1}%
u=\lim_{r\rightarrow1}v=\infty$, we obtain
\[
XY+\frac{XZ}{\delta+1}+\frac{YW}{\mu+1}=o(e^{2t})
\]
as $t\rightarrow-\infty$. Thus, eliminating $W$, we get the nonautonomous
system of order $3$
\begin{equation}
\left\{
\begin{array}
[c]{rcl}%
X_{t} & = & X\left[  X+1+Z\right]  ,\\
Y_{t} & = & Y\left[  Y+1\right]  -(\mu+1)X(Y+\frac{Z}{\delta+1})+o(e^{2t}),\\
Z_{t} & = & Z\left[  1-\delta Y-Z\right]  .
\end{array}
\right.  \label{A1p}%
\end{equation}
which appears as a perturbation of system
\begin{equation}
\left\{
\begin{array}
[c]{rcl}%
X_{t} & = & X\left[  X+1+Z\right]  ,\\
Y_{t} & = & Y\left[  Y+1\right]  -(\mu+1)X(Y+\frac{Z}{\delta+1}),\\
Z_{t} & = & Z\left[  1-\delta Y-Z\right]  .
\end{array}
\right.  \label{A1}%
\end{equation}
Moreover, by using a suitable change of variables, system (\ref{A1p})
\textit{\ reduces to a nonautonomous system of order 2}, and we can show that
the last system behaves like an autonomous one. Then we come back to the
initial system and deduce the convergence.\medskip

In the case $N\geq1$ or $a,\ b$ not necessarily equal to $0$, we first reduce
the problem to a system similar to (\ref{M1}), but nonautonomous, and we prove
that it is a perturbation of (\ref{M1}). Moreover we produce an identity that
plays the role of a first integral, allowing us to reduce to a double
perturbation of (\ref{A1}). We manage with the two perturbations in order to conclude.

\subsubsection{Steps of the proof}

Our proof relies strongly in a result due to Logemann and Ryan, see \cite{lr}.
We state it below for the convenience of the reader.

\begin{theorem}
\label{logr} \cite[Corollary 4.1]{lr} Let $h:\mathbb{R}_{+}\times
\mathbb{R}^{M}\rightarrow\mathbb{R}^{M}$ be of Carath\'{e}odory class. Assume
that there exists a locally Lipschitz continuous function $h^{\ast}%
:\mathbb{R}^{M}\rightarrow\mathbb{R}^{M}$ such that for all compact
$C\subset\mathbb{R}^{M}$ and all $\varepsilon>0$, there exists $T\geq0$ such
that
\[
\sup_{c\in C}ess\sup_{\tau\geq T}||h(\tau,x)-h^{\ast}(x)||<\varepsilon
\]
Assume that $x$ is a bounded solution of equation $x_{\tau}=h(\tau,x)$ on
$\mathbb{R}_{+}$ such that $x(0)=x_{0}.$ Then the $\omega$-limit set of $x$ is
non empty, compact and connected, and invariant under the flow generated by
$h^{\ast}$.\medskip
\end{theorem}

The proof of Theorem \ref{main} requires some important lemmas. By scaling we
still assume that\textit{\ }$R=1$.

\begin{lemma}
\label{change}Let $(u,v)$ be any fixed solution of system (\ref{sab}) in
$\left[  r_{0},1\right)  $, unbounded at\textit{\ }$1$\textit{. Let us set
}$t=\log s$, where $s=\Psi^{-1}(r)$ is defined in (\ref{tra}). Let $F,G$ be
defined by (\ref{ffgg}). Then the functions
\begin{equation}
X(t)=-\frac{su_{s}}{u}>0,\ Y(t)=-\frac{sv_{s}}{v}>0,\ Z(t)=\frac
{sF(s)v^{\delta}}{u_{s}}<0,\ W(t)=\frac{sG(s)u^{\mu}}{v_{s}}<0, \label{dei}%
\end{equation}
satisfy the (in general nonautonomous) system%
\begin{equation}
\left\{
\begin{array}
[c]{rcl}%
X_{t} & = & X\left[  X+1+Z\right]  ,\\
Y_{t} & = & Y\left[  Y+1+W\right]  ,\\
Z_{t} & = & Z\left[  1-\delta Y-Z-\alpha(t)\right]  ,\\
W_{t} & = & W\left[  1-\mu X-W-\beta(t)\right]  ,
\end{array}
\right.  \label{NA1}%
\end{equation}
where
\begin{equation}
\alpha(t)=\frac{2N-2+a}{1+(N-2)e^{t}}e^{t},\qquad\beta(t)=\frac{2N-2+b}%
{1+(N-2)e^{t}}e^{t}. \label{bab}%
\end{equation}
Moreover we recover $u,v$ by the relations
\begin{equation}
u=s^{-\gamma}F^{-\frac{1}{D}}G^{-\frac{\delta}{D}}(|Z|X)^{\frac{1}{D}%
}(|W|Y)^{\frac{\delta}{D}},\ \qquad v=s^{-\xi}F^{-\frac{\mu}{D}}G^{-\frac
{1}{D}}(|W|Y)^{\frac{1}{D}}(|Z|X)^{\frac{\mu}{D}}. \label{for}%
\end{equation}

\end{lemma}

\begin{proof}
\textit{ }Since $(u,v)$ is unbounded, (\ref{cho}) holds. We make the
substitution (\ref{tra}), which leads to system (\ref{GH}), with $F,G$ given
by (\ref{ffgg}). Clearly we can assume that $u_{s}<0$ and $v_{s}<0$ on
$\left(  0,s_{0}\right]  $, $\lim_{s\rightarrow0}\left\vert u_{s}\right\vert
=\lim_{s\rightarrow0}\left\vert v_{s}\right\vert =\lim_{s\rightarrow0}%
u=\lim_{s\rightarrow0}v=\infty$. Then we can define $X,Y,Z,W$ by (\ref{dei})
and we obtain system (\ref{NA1}) with
\[
\alpha(t)=-s\frac{F^{\prime}(s)}{F(s)},\qquad\beta(t)=-s\frac{G^{\prime}%
(s)}{G(s)};
\]
then (\ref{bab}) follows, and we deduce (\ref{for}) by straight
computation.\medskip
\end{proof}

Next we prove that system (\ref{NA1}) is a perturbation of the corresponding
autonomous system (\ref{M1}):

\begin{lemma}
\label{bounded-sol} Let $N\geq1$. Under the assumptions of Theorem \ref{main},
there exist $k>0$ and $\bar{t}<t_{0}$ such that%
\begin{equation}
1/k\leq X,Y,|Z|,|W|\leq k\quad\quad\text{for}\quad t\leq\bar{t}. \label{bb}%
\end{equation}
Moreover, setting
\begin{equation}
XY+\frac{XZ}{\delta+1}+\frac{YW}{\mu+1}=\frac{\varpi(t)}{\mu+1}, \label{hon}%
\end{equation}
we have $\varpi(t)=O(e^{t})$ as $t\rightarrow-\infty.$
\end{lemma}

\begin{proof}
We establish some integral inequalities, playing the role of a first integral,
then we use them to prove (\ref{bb}), and finally we deduce the behavior of
$\varpi$. \medskip

\noindent(i) \textit{Integral inequalities. } Let $\sigma$,\ $\theta
\in\mathbb{R}$ and set
\begin{align*}
H_{\sigma,\theta}(s)  &  =r^{2-N}\left(  u_{s}v_{s}-F(s)\frac{v^{\delta+1}%
}{\delta+1}-G(s)\frac{u^{\mu+1}}{\mu+1}-\frac{\sigma vu_{s}+\theta uv_{s}%
}{1+(N-2)s}\right) \\
&  =r^{2-N}uve^{-2t}\left(  XY+\frac{X(Z+\bar{\alpha}(t))}{\delta+1}%
+\frac{Y(W+\bar{\beta}(t))}{\mu+1}\right)  ,
\end{align*}
where
\[
\bar{\alpha}(t)=\frac{\sigma(\delta+1)s}{1+(N-2)s}\quad\text{and}\quad
\bar{\beta}(t)=\frac{\theta(\mu+1)s}{1+(N-2)s}.
\]
It can be easily verified that
\begin{equation}
H_{\sigma,\theta}^{\prime}(s)=(N-2-\sigma-\theta)u_{s}v_{s}+F(s)\frac
{v^{\delta+1}}{\delta+1}(N+a-\sigma(\delta+1))+G(s)\frac{u^{\mu+1}}{\mu
+1}(N+b-\theta(\mu+1)).
\end{equation}
By choosing first the constants $\sigma=\sigma_{1}>0$ and $\theta=\theta
_{1}>0$ large enough, we obtain that $H_{\sigma_{1},\theta_{1}}^{\prime}(s)<0$
and thus $H_{\sigma_{1},\theta_{1}}(s)\geq-C_{1}$ for some t $C_{1}>0$; next
choosing $\sigma=\sigma_{2}<0$ and $\theta=\theta_{2}<0$ and large enough in
absolute value, we obtain that $H_{\sigma_{2},\theta_{2}}^{\prime}(s)>0$ and
thus $H_{\sigma_{2},\theta_{2}}(s)\leq C_{2}$ for some $C_{2}>0$. Hence, there
exists functions $\bar{\alpha}_{i}(t),\bar{\beta}_{i}(t),i=1,2,$ which are
$O(e^{t})$ as $t\rightarrow-\infty$, and such that
\begin{align}
XY+\frac{X(Z+\bar{\alpha}_{1}(t))}{\delta+1}+\frac{Y(W+\bar{\beta}_{1}%
(t))}{\mu+1}  &  \geq-C_{1}r^{N-2}\frac{e^{2t}}{uv}\label{hin}\\
XY+\frac{X(Z+\bar{\alpha}_{2}(t))}{\delta+1}+\frac{Y(W+\bar{\beta}_{2}%
(t))}{\mu+1}  &  \leq C_{2}r^{N-2}\frac{e^{2t}}{uv}. \label{hun}%
\end{align}
(ii) \textit{Estimates from below in (\ref{bb}).} Using that $u_{ss}\leq
v^{\delta}$ and multiplying by $2u_{s}<0$ we obtain%
\[
(u_{s}^{2})_{s}\geq2v^{\delta}u_{s}=(2v^{\delta}u)_{s}-2\delta v^{\delta
-1}uv_{s}>(2v^{\delta}u)_{s}%
\]
since $v_{s}<0$ in $\left(  0,s_{0}\right]  $, hence $u_{s}^{2}-2v^{\delta
}u\leq C=(u_{s}^{2}-2v^{\delta}u)(s_{0});$ since $\lim_{s\rightarrow
0}v^{\delta}u=\infty$, it follows that $u_{s}^{2}\leq(5/2)v^{\delta}u$ on
$\left(  0,s_{1}\right]  $ , for sufficiently small $s_{1}$. Using the same
method for the second equation, we obtain from (\ref{dei}) that
\begin{equation}
X(t)\leq3\left\vert Z(t)\right\vert ,\qquad Y(t)\leq3\left\vert
W(t)\right\vert ,\qquad\text{on }\left(  -\infty,t_{1}\right]  . \label{aut}%
\end{equation}
Also, from the generalized L'H\^{o}pital's rule,
\[
\overline{\lim_{s\rightarrow0}}\frac{|Z|}{F}=\overline{\lim_{s\rightarrow0}%
}\frac{sv^{\delta}}{-u_{s}}\leq\overline{\lim_{s\rightarrow0}}\frac{\delta
sv^{\delta-1}v_{s}+v^{\delta}}{-u_{ss}}=\overline{\lim_{s\rightarrow0}}\left(
\frac{\delta Y-1}{F}\right)  ,
\]
and by symmetry
\begin{equation}
\delta\overline{\lim_{s\rightarrow0}}Y\geq1+\overline{\lim_{s\rightarrow0}%
}|Z|,\qquad\mu\overline{\lim_{s\rightarrow0}}X\geq1+\overline{\lim
_{s\rightarrow0}}|W|. \label{m23}%
\end{equation}
Suppose now that $\underline{\lim}_{t\rightarrow-\infty}X=0$. From
(\ref{m23}), $\overline{\lim}_{t\rightarrow-\infty}X\geq1/\mu$, hence there is
a sequence $\{t_{n}\}\rightarrow-\infty$ of local minima of $X$ such that
$\lim_{n\rightarrow\infty}X(t_{n})=0$, and from the definition of $X$ in
(\ref{dei}), $X(t_{n})>0$ for all $n$ sufficiently large. At each $t_{n}$ we
have that $X_{t}(t_{n})=0$ and $X_{tt}(t_{n})\geq0$. From (\ref{NA1}), using
that $X(t_{n})\not =0$, we have that $X(t_{n})+1=|Z(t_{n})|$ and hence
$|Z(t_{n})|>1$. Since $X_{tt}(t_{n})=X(t_{n})Z_{t}(t_{n})$, it follows that
$Z_{t}(t_{n})\geq0$, and thus, from the third equation in (\ref{NA1}),
$1-\delta Y(t_{n})+|Z(t_{n})|\leq0$, implying
\begin{equation}
1\leq1+|Z(t_{n})|\leq\delta Y(t_{n})+\alpha(t_{n}). \label{Yb}%
\end{equation}
From (\ref{hin}) and (\ref{aut}), we deduce
\[
Y^{2}\leq3Y(\bar{\beta}_{1}(t)+(\mu+1)X)+3(\mu+1)\frac{X\bar{\alpha}_{1}%
(t)}{\delta+1}+O(e^{2t}),
\]
hence $\lim_{n\rightarrow\infty}Y(t_{n})=0$, which contradicts (\ref{Yb}). We
conclude that $\underline{\lim}_{t\rightarrow-\infty}X>0$, and similarly for
$Y$, thus $X,Y,|Z|,|W|$ are bounded from below. \medskip

\noindent(iii) \textit{Estimates from above.} From (\ref{gna}), $s^{\gamma}u$
and $s^{\xi}v$ are bounded as $s\rightarrow0,$ thus from (\ref{for}) and
(\ref{aut}), $X^{2}Y^{2\delta}$ is bounded as $t\rightarrow\infty$. Since
$X,Y$ are bounded from below, they are bounded from above, and then also $|Z|$
and $|W|$, from (\ref{m23}), hence (\ref{bb}) holds. \medskip

\noindent(iv) \textit{Conclusion.} From (\ref{hin}), (\ref{hun}), since $X,Y$
are bounded and $|\bar{\alpha}_{i}(t)|,\ |\bar{\beta}_{i}(t)|\leq Ce^{t}$,
\begin{equation}
XY\geq\frac{X|Z|}{\delta+1}+\frac{Y|W|}{\mu+1}-C_{3}e^{t}\text{ \quad
and\quad\ }XY\leq\frac{X|Z|}{\delta+1}+\frac{Y|W|}{\mu+1}+C_{4}e^{t},
\label{han}%
\end{equation}
for some $C_{3},C_{4}>0$. Then we deduce (\ref{hon}).\medskip
\end{proof}

Next we show that a convenient combination of our solution $(X,Y,Z,W)$
satisfies a system of order 2. We have

\begin{lemma}
\label{reduction} Under the assumptions of Theorem \ref{main}, and with the
above notations, let
\begin{equation}
x(\tau)=-\frac{X(t)}{Z(t)},\quad y=-\frac{Y(t)}{Z(t)},\quad\tau=-\int
_{t}^{\bar{t}}Z(\sigma)d\sigma. \label{the}%
\end{equation}
Then $(x,y)$ lies in the region
\[
\mathcal{R}_{0}:=\{(x,y)\ |\ 1/k^{2}\leq x\leq k^{2},\quad\frac{1}{\delta
+1}+\frac{1}{2(\mu+1)k^{4}}\leq y\leq k^{2}\}
\]
for $\tau\geq\tilde{\tau}>0$, and satisfies
\begin{equation}
\left\{
\begin{array}
[c]{rcl}%
x_{\tau} & = & x(-x-\delta y+2)+\varpi_{1}(\tau)\\
y_{\tau} & = & (\frac{1}{\delta+1}-y)((\delta+1)y-(\mu+1)x)+\varpi_{2}(\tau),
\end{array}
\right.  \label{NA2}%
\end{equation}
where $\varpi_{1}(\tau)=O(e^{-K\tau})$ and $\varpi_{2}(\tau)=O(e^{-K\tau})$
for some $K>0$, as $\tau\rightarrow\infty$.
\end{lemma}

\begin{proof}
We first reduce system (\ref{NA1}) to a system of order 3: from relation
(\ref{hon}) we eliminate $W$ in the system (\ref{NA1}) and obtain
\[
\left\{
\begin{array}
[c]{rcl}%
X_{t} & = & X\left[  X+1+Z\right]  ,\\
Y_{t} & = & Y\left[  Y+1\right]  -(\mu+1)X(Y+\frac{Z}{\delta+1})+\varpi(t),\\
Z_{t} & = & Z\left[  1-\delta Y-Z-\alpha(t)\right]  ,
\end{array}
\right.
\]
which is a perturbation of system (\ref{A1}). Next, defining $x=-\frac{X}%
{Z},\quad y=-\frac{Y}{Z}$, we get the system
\[
\left\{
\begin{array}
[c]{rcl}%
x_{t} & = & Z\left[  x(2-x-\delta y)+\varpi_{1}\right] \\
y_{t} & = & Z\left[  (\frac{1}{\delta+1}-y)((\delta+1)y-(\mu+1)x)+\varpi
_{2}\right]
\end{array}
\right.
\]
with
\begin{equation}
\varpi_{1}=-\frac{\alpha(t)X}{Z^{2}}=O(e^{t}),\qquad\varpi_{2}=\frac
{\varpi(t)-\alpha(t)Y}{Z^{2}}=O(e^{t}), \label{mega}%
\end{equation}
from Lemma \ref{bounded-sol}, and then
\begin{equation}
Z_{t}=Z(1+Z(\delta y-1)+\alpha(t)). \label{zed}%
\end{equation}
and $\tau(t)$ defined by (\ref{the}) for $t\leq\bar{t}$ describes $\left[
0,\infty\right)  $ as $t$ describes $\left(  -\infty,\bar{t}\right]  $, and
$\tau/2k\leq\left\vert t\right\vert \leq2k\tau$ for $t\leq\bar{t}$. Hence we
deduce (\ref{NA2}), and the estimates of $\varpi_{1},\varpi_{2}$. Notice that
$1/k^{2}\leq x,y\leq k^{2}$ for any $\tau\geq0$ from (\ref{bb}), and from
(\ref{hon}), for $\tau\geq\tilde{\tau}>0,$%
\[
y-\frac{1}{\delta+1}=\frac{1}{XZ}(\frac{YW}{\mu+1}+o(1))\geq\frac{1}%
{2(\mu+1)k^{4}},
\]
ending the proof.\medskip
\end{proof}

Hence system (\ref{NA2}) appears as an exponential perturbation of an
autonomous system that we study now:

\begin{lemma}
Consider the system
\begin{equation}
\left\{
\begin{array}
[c]{rcl}%
x_{\tau} & = & x(2-x-\delta y)\\
y_{\tau} & = & (y-\frac{1}{\delta+1})((\mu+1)x-(\delta+1)y).
\end{array}
\right.  \label{T0}%
\end{equation}
The fixed points of system (\ref{T0}) are $O=(0,0)$, and
\[
j_{0}=\left(  0,\frac{1}{\delta+1}\right)  ,\quad\ell_{0}=\left(  \frac
{\delta+2}{\delta+1},\frac{1}{\delta+1}\right)  ,\quad m_{0}=(x_{0}%
,y_{0})=\left(  \frac{2(\delta+1)}{\mu\delta+2\delta+1},\frac{2(\mu+1)}%
{\mu\delta+2\delta+1}\right)  ,
\]
and $m_{0}$ is a sink. Any solution of the system (\ref{T0}) which stays in
the region $\mathcal{R}_{0}$ converges to the fixed point $m_{0}$ as
$\tau\rightarrow\infty$.
\end{lemma}

\begin{proof}
The point $m_{0}$ is a sink: the eigenvalues of the linearized system of
(\ref{T0}) at $m_{0}$ are the roots $\ell_{1},\ell_{2}$ of equation
\[
\ell^{2}+\frac{\delta\mu+3+2\mu+2\delta}{\mu\delta+2\delta+1}\ell+2\frac
{\mu\delta+2\mu+1}{\mu\delta+2\delta+1}=0,
\]
equivalently%
\begin{equation}
(\gamma+1)\ell^{2}+(\gamma+\xi+1)\ell+2(\xi+1)=0, \label{eqla}%
\end{equation}
and they have negative real part. Next we show that (\ref{T0}) has no limit
cycle in $\left(  0,\infty\right)  \times\left(  1/(\delta+1),\infty\right)
$. Let $\mathcal{B}=x^{p}(y-\frac{1}{\delta+1})^{-q}$, where $p,q$ are
parameters. Writing (\ref{T0}) under the form $x_{\tau}=\mathcal{F}(x,y)$,
$y_{t}=\mathcal{G}(x,y)$, we obtain
\[
\nabla\cdot(\mathcal{B}(\mathcal{F},\mathcal{G}))=\mathcal{B}_{x}x_{\tau
}+\mathcal{B}_{\eta}y_{\tau}+\mathcal{B}(\mathcal{F}_{x}+\mathcal{G}%
_{y}):=M\mathcal{B},\text{ \qquad where}%
\]%
\[
M=(\mu-1-p-q(\mu+1))x-\left(  p\delta-q(\delta+1)+3\delta+2\right)
(y-\frac{1}{\delta+1})+\frac{p(\delta+2)+q(\delta+1)+1}{\delta+1}.
\]
Choosing $q=\frac{\mu\delta+2\delta+2}{\mu\delta+2\delta+1}$ and
$p=\mu-1-q(\mu+1)$, we find that
\[
(\delta+1)M=-(\frac{\mu\delta+2\mu+1}{\mu\delta+2\delta+1}+\delta+2)<0.
\]
Hence, by the Bendixson-Dulac Theorem, system (\ref{T0}) has no limit cycle.
From the Poincar\'{e}-Bendixon Theorem, the $\omega$-limit set $\Gamma$ of any
solution of (\ref{T0}) lying in $\mathcal{R}_{0}$ is fixed point, of a union
of fixed points and connecting orbits. But $m_{0}$ is the unique fixed point
in $\mathcal{R}_{0}.$ Then any solution in $\mathcal{R}_{0}$ converges to
$m_{0}$ as $\tau\rightarrow\infty$.
\end{proof}

\begin{remark}
\label{val} It is easy to prove that there exists a connecting orbit joining
the two points $\ell_{0}$ and $m_{0},$ but it is not located in $\mathcal{R}%
_{0}$.\medskip
\end{remark}

We can now conclude.\medskip

\begin{proof}
[Proof of Theorem \ref{main}](i) \textit{Convergence for system (\ref{NA2}).}
From Proposition \ref{logr}, the $\omega$-limit set $\Sigma$ of our solution
$(x,y)$ of (\ref{NA2}) is nonempty, compact, connected and contained in
${\mathcal{R}}_{0}$, and $\Sigma=\bigcup_{\ell\in\Sigma,\tau\geq\tilde{\tau}%
}\varphi(\tau,\ell),$ where $\varphi(\tau,\ell)$ denotes the trajectory of
(\ref{T0}) such that $\varphi(\tilde{\tau},\ell)=\ell$. Since $\lim
_{\tau\rightarrow\infty}\varphi(\tau,\ell)=m_{0}$, there holds $m_{0}\in
\Sigma$. Since $m_{0}$ is a sink of (\ref{T0}), then from the standard
stability theory, see for example \cite[Theorem 3.1, page 327]{cod}, $(x,y)$
converges to $m_{0}$.\medskip

\noindent(ii) \textit{Convergence for system (\ref{NA1})}. By setting
$g(t)=1/Z$, we find from (\ref{zed}) that $g^{\prime}+(1-\alpha)g=1-\delta y,$
hence by L'H\^{o}pital's rule,
\[
\lim_{t\rightarrow-\infty}Z=\lim_{t\rightarrow-\infty}\frac{(e^{\int_{\bar{t}%
}^{t}(1-\alpha)})^{\prime}}{(ge^{\int_{\bar{t}}^{t}(1-\alpha)})^{\prime}}%
=\lim_{t\rightarrow-\infty}\frac{1-\alpha}{1-\delta y}=-\frac{\mu
\delta+2\delta+1}{\mu\delta-1}=-(1+\gamma)=Z_{0}.
\]
Hence
\[
\lim_{t\rightarrow-\infty}X=-\lim_{t\rightarrow-\infty}xZ=\frac{2(\delta
+1)}{\mu\delta-1}=\gamma=X_{0},\quad\quad\lim_{t\rightarrow-\infty}%
Y=\lim_{t\rightarrow-\infty}yZ=\frac{2(\mu+1)}{\mu\delta-1}=\xi=Y_{0}.
\]
Finally, from (\ref{hon}), we obtain $\lim_{t\rightarrow-\infty}%
W=-(1+\xi)=W_{0}$. That means $(X,Y,Z,W)$ converges to $M_{0,1}$ defined at
(\ref{poi}). Then from (\ref{for}) we deduce the estimates
\[
u(r)=A_{1}d^{-\gamma}(1+o(1)),\qquad v(r)=B_{1}d^{-\xi}(1+o(1))\text{ }%
\]
where $A_{1},B_{1}$ are given by and (\ref{A1B1}).\medskip

\noindent(iii) \textit{Expansion of }$u$\textit{ and }$v.$ We first consider
system (\ref{NA2}). Setting $x=x_{0}+\tilde{x},y_{0}+\tilde{y},$ we find a
system of the form%
\[
\left(  \tilde{x}_{\tau},\tilde{y}_{\tau}\right)  =\mathcal{A}(\tilde
{x},\tilde{y})+\mathcal{Q}(\tilde{x},\tilde{y})+\left(  \varpi_{1},\varpi
_{2}\right)
\]
where $(\tilde{x},\tilde{y})\rightarrow\left(  0,0\right)  ,$ the eigenvalues
$\ell_{1},\ell_{2}$ of $\mathcal{A}$ satisfy $\max$($\operatorname{Re}%
(\ell_{1},\ell_{2}))=-m<-1/(\gamma+1),$ and $\mathcal{Q}$ is quadratic and
$\varpi_{1}(\tau),\varpi_{2}(\tau)=O(e^{-K\tau}).$ There exists an euclidian
structure with a scalar product where $\langle\mathcal{A}(\tilde{x},\tilde
{y}),(\tilde{x},\tilde{y})\rangle$ $\leq-m\left\Vert (\tilde{x},\tilde
{y})\right\Vert ^{2}$ . Then the function $\tau\mapsto\eta(\tau)=\left\Vert
(\tilde{x},\tilde{y})\right\Vert \left(  \tau\right)  $ satisfies an
inequality of the type $\eta_{\tau}\leq-(m-\varepsilon)\eta+Ce^{-K\tau}$ for
any $\varepsilon>0$ and $\tau$ large enough. Then
\begin{equation}
\eta(\tau)=O(e^{-K\tau})+O(e^{-(m-\varepsilon)\tau}). \label{hot}%
\end{equation}
Then the convergence of $(x,y)$ to $(x_{0},y_{0})$ is exponential. From
(\ref{zed}), the convergence of $Z$ to $Z_{0}$ is exponential. Writing $\tau$
under the form
\[
\tau=\overline{c}+Z_{0}t+\int_{t}^{\infty}(Z_{0}-Z),
\]
we deduce that $\tau=\overline{c}+Z_{0}t+O(e^{kt})$ for some $k>0.$ From
(\ref{mega}) we obtain that $\varpi_{1},\varpi_{2}=$ $O(e^{-K_{0}\tau})$ with
$K_{0}=1/\left\vert Z_{0}\right\vert ;$ taking $K=K_{0}=1/(\gamma+1)$ in
(\ref{hot}), we find that $\eta(\tau)=O(e^{-K_{0}\tau})=O(e^{t}),$ because
$m>K_{0}.$ Then from (\ref{zed}) we deduce that $\left\vert Z-Z_{0}\right\vert
=O(e^{t}),$ and then from (\ref{the}), $\left\vert X-X_{0}\right\vert
+\left\vert Y-Y_{0}\right\vert =O(e^{t}),$ and in turn $\left\vert
W-W_{0}\right\vert =O(e^{t})$ from (\ref{hon}). Finally we come back to $u$
and $v$ by means of (\ref{for}): recalling that $s=e^{t}$ and $r=1+O(s)$ as
$s\rightarrow0,$ we deduce that
\[
u(r)=A_{1}s^{-\gamma}(1+O(s)),\qquad v(r)=B_{1}s^{-\xi}(1+O(s))\text{ }%
\]
and the expansion (\ref{esti}) follows from (\ref{tra}).\medskip
\end{proof}

\begin{proof}
[Proof of Corollary \ref{bila}]Let $u$ be a radial solution of (\ref{bih}).
Then $u$ and $v=\Delta u$ satisfy
\[
\left\{
\begin{array}
[c]{l}%
\Delta u=v\\
\Delta v=|x|^{b}\left\vert u\right\vert ^{\mu}%
\end{array}
\right.
\]
and then $u(r)>0$ in $(r_{0},R)$ and $u(R)=\infty$. Integrating twice the
second equation in this system, we have that $\lim\limits_{r\rightarrow
R}v(r)=\infty$ and Theorem \ref{main} applies.
\end{proof}

\subsection{The set of initial data for blow up}

Here we suppose $a=b=0$. By scaling, for any $\rho>0$ there exists solutions
which blow up at $\rho$. Let us call $\rho(u_{0},v_{0})$ the blow-up radius of
a regular solution with initial data $(u_{0},v_{0})$. From (\ref{scaling}), we
find
\[
\rho(\lambda^{\gamma}u_{0},\lambda^{\xi}v_{0})=\lambda^{-1}\rho(u_{0},v_{0}).
\]
Then for any $(u_{0},v_{0})\in S^{1}$ there is a unique $\lambda$ such that
$\rho(\lambda^{\gamma}u_{0},\lambda^{\xi}v_{0})=1.$ Thus there exist
infinitely many solutions blowing up at $R=1$, including in particular two
unique solutions with respective initial data $(\bar{u}_{0},0)$ and
$(0,\bar{v}_{0})$. Using monotonicity properties, it was shown in \cite{gmls2}
that the set
\[
\mathcal{S}=\left\{  (u_{0},v_{0})\in\left[  0,\infty\right)  \times\left[
0,\infty\right)  :\text{ }\lim_{r\rightarrow1}u=\lim_{r\rightarrow1}%
v=\infty\right\}
\]
is contained in $\left[  0,\bar{u}_{0}\right]  \times\left[  0,\bar{v}%
_{0}\right]  $. Next we give some properties of $\mathcal{S}$ extending some
results of \cite{gmls2} to higher dimensions.

\begin{proposition}
\label{set}Let $N\geq1$. If $\min\left\{  \delta,\mu\right\}  \geq1$, then
$\mathcal{S}$ is a \textit{simple curve joining the two points }$(\bar{u}%
_{0},0)$\textit{\ and }$(0,\bar{v}_{0})$.
\end{proposition}

\begin{proof}
We claim that the mapping $(u_{0},v_{0})\in$ $\left[  0,\infty\right)
\times\left[  0,\infty\right)  \backslash\left\{  0,0\right\}  \longmapsto
\rho(u_{0},v_{0})$ is continuous. As in \cite{DLS} this will follow from our
global estimates.\medskip

(i) \textit{The function }$\rho$\textit{ is lower semi-continuous.} Indeed the
local existence is obtained by the fixed point theorem of a strict
contraction, since $\min\{\delta,\mu\}\geq1$, then we have local continuous
dependence of the initial conditions, even if $u_{0}=0$ or $v_{0}=0$, and the
result follows classically.\medskip

(ii) \textit{The function }$\rho$\textit{ is upper semi-continuous.} We can
start from a point $r_{0}>0$ instead of $0$. We prove that for any positive
$(\tilde{u}_{0},\tilde{v}_{0})$, considering any solution $(\tilde{u}%
,\tilde{v})$ equal to $(\tilde{u}_{0},\tilde{v}_{0})$ at $r_{0}$, with blow-up
point $\tilde{\rho}$, for any $\tilde{r}>\tilde{\rho}$, any solution $(u,v)$
starting from $r_{0}$ with data sufficiently close to $(\tilde{u}_{0}%
,\tilde{v}_{0})$, blows up before $\tilde{r}:$ suppose that it is false, then
there exists a sequence of positive solutions $(u_{n},v_{n})$, with data
$(\tilde{u}_{n},\tilde{v}_{n})$ at $r_{0}$, tending to $(\tilde{u}_{0}%
,\tilde{v}_{0})$, increasing, and blowing up at $\rho_{n}\geq$ $\tilde{r}$. We
can assume $\tilde{r}=1$. Making the change of variables (\ref{tra}) we get
solutions of system (\ref{GH}) in $\left(  0,s_{0}\right]  $, satisfying
$C_{0}=C_{0}(r_{0},N,a,b)$
\[
\left\{
\begin{array}
[c]{c}%
-u_{ss}+C_{0}v^{\delta}\leq0\\
-v_{ss}+C_{0}u^{\mu}\leq0
\end{array}
\right.
\]
with $u$ and $v$ decreasing. In fact estimates (\ref{oe}) hold with a
universal constant, in any $B(0,k)\backslash\left\{  0\right\}  $
$\subset\Omega$ such that the mean values of $u$ and $v$ on $\partial B(0,r)$
are strictly monotone. Then there exists a constant $C=C(C_{0},N,\delta,\mu)$
such that
\[
u_{n}(s)\leq Cs^{-\gamma},\quad v_{n}(s)\leq Cs^{-\xi}\quad\text{for }s\leq
s_{0},
\]
that means
\[
u_{n}(r)\leq C(r^{2-N}-1)^{-\gamma},\quad v_{n}(r)\leq C(r^{2-N}-1)^{-\xi
}\quad\text{for }r\in\left[  r_{0},1\right)  .
\]
Passing to the limit we find that $u,v$ are bounded at the point $\tilde{\rho
}<1$, which is contradictory. Then the claim is proved. Thus $\mathcal{S}$ is
a curve with
\[
(u_{0},v_{0})=\left[  \rho^{\gamma}\left(  \cos\theta,\sin\theta\right)
\cos\theta,\rho^{\xi}\left(  \cos\theta,\sin\theta\right)  \sin\theta\right]
,\qquad\theta\in\left[  0,\pi/2\right]  ,
\]
as a parametric representation.
\end{proof}

\section{Behavior of system (\ref{sab}) near the origin\label{orig}}

\subsection{Formulation as a dynamical system}

In \cite{BGi} the authors study general quasilinear elliptic systems, and in
particular the system%

\begin{equation}
\left\{
\begin{array}
[c]{c}%
-\Delta u=-(u_{rr}+\frac{N-1}{r}u_{r})=\varepsilon_{1}r^{a}v^{\delta},\\
-\Delta v=-(v_{rr}+\frac{N-1}{r}v_{r})=\varepsilon_{2}r^{b}u^{\mu},
\end{array}
\right.  \label{gen}%
\end{equation}
where $\varepsilon_{1}=\pm1$, $\varepsilon_{2}=\pm1$. Near any point $r$ where
$u(r)\neq0,u^{\prime}(r)\neq0$ and $v(r)\neq0$, $v^{\prime}(r)\neq0,$ they
define%
\begin{equation}
X(t)=-\frac{ru_{r}}{u},\qquad Y(t)=-\frac{rv_{r}}{v},\qquad Z(t)=-\varepsilon
_{1}\frac{r^{1+a}v^{\delta}}{u_{r}},\qquad W(t)=-\varepsilon_{2}\frac
{r^{1+b}u^{\mu}}{v_{r}}, \label{xt}%
\end{equation}
with $t=\ln r$, so system (\ref{gen}) becomes%
\begin{equation}
\left\{
\begin{array}
[c]{rcl}%
X_{t} & = & X\left[  X-(N-2)+Z\right]  ,\\
Y_{t} & = & Y\left[  Y-(N-2)+W\right]  ,\\
Z_{t} & = & Z\left[  N+a-\delta Y-Z\right]  ,\\
W_{t} & = & W\left[  N+b-\mu X-W\right]  .
\end{array}
\right.  \label{mn}%
\end{equation}
One recovers $u$ and $v$ by the formulas%
\begin{equation}
u\mathbf{=}r^{-\gamma_{a,b}}\mathbf{(}\left\vert ZX\right\vert )^{1/D}%
\mathbf{(}\left\vert WY\right\vert )^{\delta/D},\mathbf{\qquad}v\mathbf{=}%
r^{-\xi_{a,b}}\left\vert WY\right\vert )^{1/D}\left\vert ZX\right\vert
^{\mu/D}, \label{form}%
\end{equation}
and we notice the relations $\gamma_{a,b}+2+a=\delta\xi_{a,b}$ and $\xi
_{a,b}+2+b=\mu\gamma_{a,b}.$\medskip

As mentioned in \cite{BGi}, system (\ref{mn}) is independent of $\varepsilon
_{i}$, $i=1,2$, and thus it allows to study system (\ref{gen}) in a unified
way. In our case $\varepsilon_{1}=\varepsilon_{2}=-1$, then $XZ=-r^{a+2}%
v^{\delta}/u$ and $YW=r^{b+2}u^{\mu}/v,$ thus we are led to study (\ref{mn})
in the region
\[
\mathcal{R}=\{(X,Y,Z,W)\ |\ XZ\leq0,\ YW\leq0\}.
\]

This system is \textit{quadratic}, and it admits four invariant hyperplanes:
$X=0,Y=0,Z=0,W=0$. The trajectories located on these hyperplanes do not
correspond to a solution of system (\ref{gen}), and they are called
\textit{\textit{nonadmissible}}. System (\ref{mn}) has sixteen fixed points,
including $O=(0,0,0,0)$. The main one is
\[
M_{0}=\left(  X_{0},Y_{0},Z_{0},W_{0}\right)  =\left(  \gamma_{a,b},\xi
_{a,b},N-2-\gamma_{a,b},N-2-\xi_{a,b}\right)  ,
\]
which is interior to $\mathcal{R}$ whenever (\ref{cde}) holds; it corresponds
to the particular solution $(u^{\ast},v^{\ast})$ given in (\ref{zero}). Among
the other fixed points, as we see below,%
\begin{align*}
N_{0}  &  =(0,0,N+a,N+b),\\
R_{0}  &  =\left(  0,-(2+b),N+a+(2+b)\delta,N+b\right)  ,\quad S_{0}=\left(
-(2+a),0,N+a,N+b+(2+a)\mu\right)  ,
\end{align*}
are linked to the regular solutions, and
\[
A_{0}=(N-2,N-2,0,0),\,G_{0}=(N-2,0,0,N+b-(N-2)\mu),\,H_{0}%
=(0,N-2,N+a-(N-2)\delta,0),
\]%
\begin{align*}
P_{0}  &  =(N-2,(N-2)\mu-2-b,0,(N+b-(N-2)\mu)),\\
\ Q_{0}  &  =((N-2)\delta-2-a,N-2,N+a-(N-2)\delta,0),
\end{align*}
and $M_{0}$are linked to the large solutions near $0.$ Notice that
$P_{0}\not \in \mathcal{R}$ for $\frac{2+b}{N-2}<\mu<\frac{N+b}{N-2}$ and
$Q_{0}\not \in \mathcal{R}$ for $\frac{2+a}{N-2}<\delta<\frac{N+a}{N-2}$. We
are not concerned by the other fixed points
\[
I_{0}=(N-2,0,0,0),J_{0}\!=(0,N-2,0,0),K_{0}\!=(0,0,N+a,0),L_{0}%
\!=(0,0,0,N+b),
\]
which correspond to non admissible solutions, from \cite{BGi}, and
\[
C_{0}=\left(  0,-(2+b),0,N+b\right)  ,\qquad D_{0}=\left(
-(2+a),0,N+a,0\right)  ,
\]
which can be shown as non admissible as $t\rightarrow-\infty.\medskip$

\subsection{Regular solutions}

First we give an alternative proof of Proposition \ref{ereg}.

\begin{proposition}
\label{ereg2} Assume (\ref{abn}) and $D\neq0.$ Then a solution $(u,v)$ is
regular\textrm{ }with initial data $(u_{0},v_{0}),u_{0},v_{0}>0$ (resp.
$(u_{0},0)$,$u_{0}>0,$ resp. $(0,v_{0}),v_{0}>0$), if and only the
corresponding solution $(X,Y,Z,W)$ converges to $N_{0}$ (resp. $R_{0},$ resp.
$S_{0}$) as $t\rightarrow-\infty$. For any $u_{0},v_{0}\geq0,$ not both $0$,
there exists a unique local regular solution $(u,v)$ with initial data
$(u_{0},v_{0}).$
\end{proposition}

\begin{proof}
The proof in the case $u_{0},v_{0}>0$ is done in \cite[Proposition 4.4]{BGi}.
Suppose $u_{0}>0=v_{0},$ and consider any regular solution $(u,v)$ with
initial data $(u_{0},0).$ We find
\begin{align*}
v^{\prime}  &  =\frac{u_{0}^{\mu}}{N+b}r^{1+b}(1+o(1)),\qquad v=\frac
{u_{0}^{\mu}}{(N+b)(2+b)}r^{2+b}(1+o(1)),\\
(r^{N-1}u^{\prime})^{\prime}  &  =\frac{u_{0}^{\mu^{\delta}}%
r^{N-1+a+(2+b)\delta}}{((N+b)(2+b))^{\delta}}(1+o(1)),\ u^{\prime}=\frac
{u_{0}^{\delta\mu}r^{1+a+(2+b)\delta}}{((N+b)(2+b))^{\delta}(N+a+(2+b)\delta
)}(1+o(1));
\end{align*}
then from (\ref{xt}) the corresponding trajectory $(X,Y,Z,W)$ converges to
$R_{0}$ as $t\rightarrow-\infty.$ Next we show that there exists a unique
trajectory converging to $R_{0}.\ $ We write
\[
R_{0}=\left(  0,\bar{Y},\bar{Z},\bar{W}\right)  =\left(
0,-(2+b),N+a+(2+b)\delta,N+b\right)  .
\]
Under our assumptions it lies in $\mathcal{R}$. Setting $Y=\bar{Y}+\tilde
{Y},Z=\bar{Z}+\tilde{Z},W=\bar{W}+\tilde{W},$ the linearization at $R_{0}$
gives%
\begin{equation}
X_{t}=\lambda_{1}X,\quad\tilde{Y}_{t}=\bar{Y}\left[  \tilde{Y}+\tilde
{W}\right]  ,\quad Z_{t}=\bar{Z}\left[  -\delta\tilde{Y}-\tilde{Z}\right]
,\quad W_{t}=\bar{W}\left[  -\mu X-\tilde{W}\right]  ;\nonumber
\end{equation}
the eigenvalues are
\[
\lambda_{1}=2+a+\delta(2+b)>0,\quad\lambda_{2}=-(2+b)<0,\quad\lambda_{3}%
=-\bar{Z}<0,\quad\lambda_{4}=-(N+b)<0.
\]
The unstable manifold $\mathcal{V}_{u}$ has dimension 1 and $\mathcal{V}%
_{u}\cap\left\{  X=0\right\}  =\emptyset$, hence there exist precisely one
admissible trajectory such that $X<0$ and $Z>0$. Moreover it satisfies
\[
\lim_{t\rightarrow-\infty}e^{-\lambda_{1}t}X=C_{1}>0,\lim_{t\rightarrow
-\infty}Y=\bar{Y},\lim_{t\rightarrow-\infty}Z=\bar{Z},\lim_{t\rightarrow
-\infty}W=\bar{W}.
\]
Then from (\ref{form}) $u$ has a positive limit $\overline{u}_{0}$, and
$v\mathbf{=}O(e^{2t}),$ thus $v$ tends to 0; then $(u,v)$ is regular with
initial data $\left(  \overline{u}_{0},0\right)  $. By (\ref{scaling}) we
obtain existence for any $(u_{0},0)$ and the uniqueness still holds. Similarly
the solutions with initial data $(0,v_{0})$ correspond to $S_{0}.$
\end{proof}

\subsection{Local existence of large solutions near $0$}

Next we prove the existence of different types of local solutions large at 0,
by linearization around the fixed points $A_{0},G_{0},H_{0},P_{0},Q_{0}$. For
simplicity we do not consider the limit cases, where one of the eigenvalues of
the linearization is 0, corresponding to behaviors of $u,v$ of logarithmic
type. All the following results extend by symmetry, after exchanging
$u,\delta,a,\gamma_{a,b}$ and $v,\mu,b,\xi_{a,b}$. \medskip

\begin{proposition}
\label{ex1} Assume $N>2$.\medskip\ 

(i) If $\delta<\frac{N+a}{N-2}$ and $\mu<\frac{N+b}{N-2}$, then there exist
solutions $(u,v)$ to (\ref{sab}) such that
\begin{equation}
\lim_{r\rightarrow0}r^{N-2}u=\alpha>0,\qquad\lim_{r\rightarrow0}r^{N-2}%
v=\beta>0. \label{ao}%
\end{equation}
If $\delta>\frac{N+a}{N-2}$ or $\mu>\frac{N+b}{N-2}$, there exist no such
solutions. \medskip

(ii) Let $\gamma_{a,b}>N-2$ and let $\mu<\frac{2+b}{N-2}$ or $\mu>\frac
{N+b}{N-2}$. Then there exist solutions $(u,v)$ of (\ref{sab}) such that
\begin{equation}
\lim_{r\rightarrow0}r^{N-2}u=\alpha>0,\qquad\lim_{r\rightarrow0}%
r^{(N-2)\mu-(2+b)}v=\beta(\alpha)>0, \label{po}%
\end{equation}
with $\beta(\alpha)=\alpha^{\mu}/((N-2)\mu-N-b)((N-2)\mu-2-b)$. If
$\gamma_{a,b}<N-2$, there exist no such solutions. \medskip

(iii) If $\mu<\frac{2+b}{N-2}$ then there exist solutions $(u,v)$ of
(\ref{sab}) such that%
\begin{equation}
\lim_{r\rightarrow0}r^{N-2}u=\alpha>0,\qquad\lim_{r\rightarrow0}v=\beta>0.
\label{go}%
\end{equation}
If $\mu>\frac{2+b}{N-2}$ there exist no such solutions.
\end{proposition}

\begin{proof}
(i) We study the behaviour of the solutions of (\ref{mn}) near $A_{0}$ as
$t\rightarrow-\infty$. The linearization at $A_{0}$ gives, with $X=N-2+\tilde
{X},Y=N-2+\tilde{Y},$%
\[
\tilde{X}_{t}=(N-2)\left[  \tilde{X}+Z\right]  ,\quad\tilde{Y}_{t}%
=(N-2)\left[  \tilde{Y}+W\right]  ,\quad Z_{t}=\lambda_{3}Z,\quad
W_{t}=\lambda_{4}W,
\]
with eigenvalues
\[
\lambda_{1}=\lambda_{2}=(N-2)>0,\;\lambda_{3}=N+a-(N-2)\delta,\;\lambda
_{4}=N+b-(N-2)\mu.
\]
If $\delta<\frac{N+a}{N-2}$ and $\mu<\frac{N+b}{N-2}$, then we have
$\lambda_{3},\lambda_{4}>0;$ the unstable manifold $\mathcal{V}_{u}$ has
dimension $4,$ then there exists an infinity of trajectories converging to
$A_{0}$ as $t\rightarrow-\infty$, interior to $\mathcal{R},$ then admissible,
with $Z,W<0$. The solutions satisfy $\lim\limits_{t\rightarrow-\infty
}e^{-\lambda_{3}t}Z=Z_{0}<0$ and $\lim\limits_{t\rightarrow-\infty}%
e^{-\lambda_{4}t}W=W_{0}<0$, with $\lim\limits_{t\rightarrow-\infty}%
X=\lim\limits_{t\rightarrow-\infty}Y=N-2.$ Hence from (\ref{form}), the
corresponding solutions $(u,v)$ of (\ref{sab}) satisfy (\ref{ao}). If
$\delta>\frac{N+a}{N-2}$ or $\mu>\frac{N+b}{N-2}$, then $\lambda_{3}<0$ or
$\lambda_{4}<0$, respectively, and $\mathcal{V}_{u}$ has at most dimension
$3$, and it satisfies $Z=0$ or $W=0$ respectively. Therefore there is no
admissible trajectory converging at $-\infty$.\medskip\ 

(ii) Here we study the behaviour near $P_{0}$. Setting $P_{0}=\left(
N-2,Y_{\ast},0,W_{\ast}\right)  $, with
\[
Y_{\ast}=(N-2)\mu-2-b,\qquad W_{\ast}=N+b-(N-2)\mu,
\]
the linearization at $P_{0}$ gives, with $X=N-2+\tilde{X},Y=Y_{\ast}+\tilde
{Y}$, $W=W_{\ast}+\tilde{W}$,
\[
\tilde{X}_{t}=(N-2)\left[  \tilde{X}+Z\right]  ,\quad\tilde{Y}_{t}=Y_{\ast
}\left[  \tilde{Y}+\tilde{W}\right]  ,\quad Z_{t}=\lambda_{3}Z,\quad\tilde
{W}_{t}=W_{\ast}\left[  -\mu\tilde{X}-\tilde{W}\right]  .
\]
By direct computation we obtain that the eigenvalues are
\[
\lambda_{1}=N-2>0,\quad\lambda_{2}=Y_{\ast},\quad\lambda_{3}=N+a-\delta
Y_{\ast}=D(\gamma_{a,b}-(N-2)),\quad\lambda_{4}=-W_{\ast}.
\]
Assume first that $\gamma_{a,b}>N-2$. Then $\lambda_{3}>0$. If $\mu>\frac
{N+b}{N-2}$, then also $\lambda_{2},\lambda_{4}>0$ and thus $\mathcal{V}_{u}$
has dimension $4$, then there exist an infinity of admissible trajectories,
with $Z<0,$ converging as $t\rightarrow-\infty$. If $\mu<\frac{2+b}{N-2}$,
then $\lambda_{2},\lambda_{4}<0$, thus $\mathcal{V}_{u}$ has dimension $2$,
and $\mathcal{V}_{u}\cap\left\{  Z=0\right\}  $ has dimension $1$, thus there
also exist an infinity of admissible trajectories with $Z<0$ converging when
$t\rightarrow-\infty$. Then $\lim\limits_{t\rightarrow-\infty}e^{-\lambda
_{3}t}Z=C_{3}<0$, $\lim\limits_{t\rightarrow-\infty}X=N-2$, $\lim
\limits_{t\rightarrow-\infty}Y=Y_{\ast}$ and $\lim\limits_{t\rightarrow
-\infty}W=W_{\ast}$, thus (\ref{form}), $(u,v)$ satisfy (\ref{po}). If
$\gamma_{a,b}<N-2$, then $\lambda_{3}<0$ and $\mathcal{V}_{u}=\mathcal{V}%
_{u}\cap\left\{  Z=0\right\}  $ and there is no admissible trajectory
converging when $t\rightarrow-\infty$. \medskip

(iii) We consider the behaviour near $G_{0}$. The linearization at $G_{0}$
gives, with $X=N-2+\tilde{X},W=N+b-(N-2)\mu+\tilde{W}$,
\begin{align*}
\tilde{X}_{t}  &  =(N-2)\left[  \tilde{X}+Z\right]  ,\quad Y_{t}%
=(2+b-(N-2)\mu)Y,\\
Z_{t}  &  =(N+a)Z,\quad W_{t}=(N+b-(N-2)\mu)\left[  -\mu\tilde{X}-\tilde
{W}\right]  ,
\end{align*}
and the eigenvalues are
\[
\lambda_{1}=N-2>0,\;\lambda_{2}=2+b-(N-2)\mu,\;\lambda_{3}=N+a>0,\;\lambda
_{4}=(N-2)\mu-N-b.
\]
If $\mu<\frac{2+b}{N-2}$, then $\lambda_{2},\lambda_{4}<0$. Then
$\mathcal{V}_{u}$ has dimension 3, and $\mathcal{V}_{u}\cap\left\{
Y=0\right\}  $ and $\mathcal{V}_{u}\cap\left\{  Z=0\right\}  $ have dimension
2. This implies that ${\mathcal{V}}_{u}$ must contain admissible trajectories
such that $X>0$ (because $N-2>0$), $Y<0$, $Z<0$ and $W>0$ (because
$N+b-(N-2)\mu>0$). Clearly, $\lim\limits_{t\rightarrow-\infty}X=N-2$ and
$\lim\limits_{t\rightarrow-\infty}W=N+b-(N-2)\mu>0$. Moreover, $\lim
\limits_{t\rightarrow-\infty}e^{-\lambda_{2}t}Y=C_{2}<0$ and $\lim
\limits_{t\rightarrow-\infty}e^{-\lambda_{3}t}Z=C_{3}<0$, thus (\ref{go})
follows from (\ref{form}). Let now $\mu>\frac{2+b}{N-2}$, so that $\lambda
_{2}<0$. If $\mu<\frac{N+b}{N-2}$, then $\lambda_{4}<0$, $\mathcal{V}_{u}$ has
dimension 2, and also $\mathcal{V}_{u}\cap\left\{  Y=0\right\}  $, hence
$\mathcal{V}_{u}=\mathcal{V}_{u}\cap\left\{  Y=0\right\}  $, and there exists
no admissible trajectory. If $\mu>\frac{N+b}{N-2}$, then $\lambda_{4}>0$,
$\mathcal{V}_{u}$ has dimension 3 and also $\mathcal{V}_{u}\cap\left\{
Y=0\right\}  $, there is no admissible trajectory.\medskip
\end{proof}

\begin{remark}
If $\mu>\frac{N+b}{N-2},$ in (ii) the two functions $u,v$ are large near $0.$
If $\mu<\frac{2+b}{N-2},$ then $u$ is large near $0$ and $v$ tends to
$0.\medskip$
\end{remark}

Next we study the behavior near $M_{0},$ which is the most interesting one.

\begin{proposition}
\label{mani}Assume $N\geq1$ and (\ref{cde}). Then (up to a scaling) there
exist infinitely many solutions defined near $r=0$ such that
\[
\lim_{r\rightarrow0}r^{\gamma_{a,b}}u=A_{N},\lim_{r\rightarrow0}r^{\xi_{a,b}%
}v=B_{N}.
\]

\end{proposition}

\begin{proof}
Setting $X=X_{0}+\tilde{X},Y=Y_{0}+\tilde{Y},Z=Z_{0}+\tilde{Z},W=W_{0}%
+\tilde{W}$, the linearized system is%
\[
\left\{
\begin{array}
[c]{rcl}%
\tilde{X}_{t} & = & X_{0}(\tilde{X}+\tilde{Z}),\\
\tilde{Y}_{t} & = & Y_{0}(\tilde{Y}+\tilde{W}),\\
\tilde{Z}_{t} & = & Z_{0}(-\delta\tilde{Y}-\tilde{Z}),\\
\tilde{W}_{t} & = & W_{0}(-\mu\tilde{X}-\tilde{W}).
\end{array}
\right.
\]
As described in \cite{BGi}, the eigenvalues are the roots $\lambda_{1}%
,\lambda_{2},\lambda_{3},\lambda_{4}$, of the characteristic polynomial%
\begin{align}
f(\lambda)  &  =\det\left(
\begin{array}
[c]{cccc}%
X_{0}-\lambda & 0 & X_{0} & 0\\
0 & Y_{0}-\lambda & 0 & Y_{0}\\
0 & \delta\left\vert Z_{0}\right\vert  & \left\vert Z_{0}\right\vert -\lambda
& 0\\
\mu\left\vert W_{0}\right\vert  & 0 & 0 & \left\vert W_{0}\right\vert -\lambda
\end{array}
\right) \nonumber\\
&  =(\lambda-X_{0})(\lambda+Z_{0})(\lambda-Y_{0})(\lambda+W_{0})-\delta\mu
X_{0}Y_{0}Z_{0}W_{0}, \label{vale}%
\end{align}
where we recall that $X_{0},Y_{0}>0$ and $Z_{0},W_{0}<0$. We write $f$ in the
form
\[
f(\lambda)=\lambda^{4}+E_{0}\lambda^{3}+F_{0}\lambda^{2}+G_{0}\lambda-H_{0},
\]
with
\[
\left\{
\begin{array}
[c]{rcl}%
E_{0} & = & Z_{0}-X_{0}+W_{0}-Y_{0},\\
F_{0} & = & (Z_{0}-X_{0})(W_{0}-Y_{0})-X_{0}Z_{0}-Y_{0}W_{0},\\
G_{0} & = & -Y_{0}W_{0}(Z_{0}-X_{0})-X_{0}Z_{0}(W_{0}-Y_{0}),\\
H_{0} & = & DX_{0}Y_{0}Z_{0}W_{0}.
\end{array}
\right.
\]
We note that $E_{0}<0,\ F_{0}>0$ and $2G_{0}=-E_{0}\left[  Y_{0}Z_{0}%
+X_{0}W_{0}\right]  <0$. From (\ref{dd}) we have $H_{0}>0$, hence $\lambda
_{1}\lambda_{2}\lambda_{3}\lambda_{4}<0$. Hence there exist two real roots
$\lambda_{3}<0<\lambda_{4}$, with
\[
\lambda_{4}>\max(\{X_{0},Y_{0},\left\vert Z_{0}\right\vert ,\left\vert
W_{0}\right\vert \}
\]
from (\ref{vale}), and two roots $\lambda_{1},\lambda_{2}$, which may be real
or complex. From the form of $f(\lambda)$ in (\ref{vale}), we also see easily
that if the roots $\lambda_{1},\lambda_{2}$ are real, they are positive. Next
we claim that $\operatorname{Re}\lambda_{1}>0$. Suppose $\operatorname{Re}%
\lambda_{1}=0$. Then $f(i\operatorname{Im}\lambda_{1})=0,$ then $G_{0}%
^{2}=E_{0}F_{0}G_{0}+E_{0}^{2}H_{0}$, and thus, dividing by $E_{0}$,
\[
0=G_{0}^{2}-E_{0}F_{0}G_{0}+E_{0}^{2}H_{0}=\frac{E_{0}^{2}}{4}\left(  \left[
Y_{0}Z_{0}+X_{0}W_{0}\right]  ^{2}+2\left[  Y_{0}Z_{0}+X_{0}W_{0}\right]
F_{0}-4H_{0}\right)  ,
\]
hence $\left[  Y_{0}Z_{0}+X_{0}W_{0}+F_{0}\right]  ^{2}=F_{0}^{2}+4H_{0}%
>F_{0}^{2};$ but
\[
Y_{0}Z_{0}+X_{0}W_{0}+F_{0}=(X_{0}-W_{0})(Y_{0}-Z_{0})\in(0,F_{0})
\]
which is a contradiction. Since $\operatorname{Re}\lambda_{1}$ is a continuous
function of $(\delta,\mu)$, it is sufficient to find a value $(\mu,\delta)$
satisfying (\ref{cde}) for which it is positive. Taking $\delta=\mu$, the
equation in $\lambda$ reduces to two equations of order $2$:%
\begin{align*}
f(\lambda)  &  =(\lambda-X_{0})^{2}(\lambda-\left\vert Z_{0}\right\vert
)^{2}-\delta^{2}X_{0}^{2}Z_{0}^{2}\\
&  =\left[  \lambda^{2}-(X_{0}+\left\vert Z_{0}\right\vert )\lambda
-(\delta-1)X_{0}\left\vert Z_{0}\right\vert \right]  \left[  \lambda
^{2}-(X_{0}+\left\vert Z_{0}\right\vert )\lambda+(1+\delta)X_{0}\left\vert
Z_{0}\right\vert \right]  ,
\end{align*}
and $X_{0}+\left\vert Z_{0}\right\vert >0$, thus the claim is proved. Then
$\mathcal{V}_{u}$ has dimension $3$ and $\mathcal{V}_{s}$ has dimension $1$.
Hence the result follows.\medskip
\end{proof}

\begin{remark}
In the case $N=1$, two roots are explicit: $\lambda_{3}=-1,\lambda
_{4}=2+\gamma+\xi$, and $\lambda_{1},\lambda_{2}$ are the roots of equation
\begin{equation}
\lambda^{2}-(1+\gamma+\xi)\lambda+2(1+\gamma)(1+\xi)=0. \label{vvv}%
\end{equation}
The 4 roots are real if $(1+\gamma+\xi)^{2}-8(1+\gamma)(1+\xi)\geq0$, that
means
\[
(\delta\mu+3+2\mu+2\delta)^{2}-8(\mu\delta+2\delta+1)(\mu\delta+2\mu+1)\geq0,
\]
which is not true for $\delta=\mu$, but is true for example when $\delta/\mu$
is large enough. The roots of equation (\ref{vvv}) and the roots of equation
(\ref{eqla}) relative to the linearization of system (\ref{T0}) at $m_{0}$ are
linked by the relations $\ell_{1}=\lambda_{1}/\left\vert Z_{0}\right\vert ,$
$\ell_{2}=\lambda_{2}/\left\vert Z_{0}\right\vert .$ Indeed $M_{0}=M_{0,1}%
\ $defined at (\ref{poi}) satisfies relation (\ref{hon}) with $\varpi=0$, thus
$(X_{0},Y_{0},Z_{0})$ is a fixed point of system (\ref{A1}) and the
linearization of (\ref{A1}) at this point gives the eigenvalues $-1,\lambda
_{1},\lambda_{2}$. The point $m_{0}$ is the image of $(X_{0},Y_{0},Z_{0})$ by
the transformation (\ref{the}), which divides the eigenvalues by $\left\vert
Z_{0}\right\vert $, due to the change in time $t\mapsto\tau$.
\end{remark}

\subsection{Global results}

Here we prove our second main result.\medskip

\begin{proof}
[Proof of Theorem \ref{global}]\noindent From the proof of Proposition
\ref{mani}, the linearization at $M_{0}$ admits a unique real eigenvalue
$\lambda_{3}<0$. From (\ref{vale}) a generating eigenvector $(u_{1}%
,u_{2},u_{3},u_{4})$ satisfies $u_{1}u_{3}<0$ and $u_{2}u_{4}<0$, and hence it
is of the form $\vec{u}=(-\alpha^{2},-\beta^{2},\sigma^{2},\rho^{2})$, or
$-\vec{u}$. There exist precisely two trajectories $\mathcal{T}_{\vec{u}}$ and
$\mathcal{T}_{-\vec{u}}$ converging to $M_{0}$ as $t\rightarrow\infty$ and the
convergence of $X,Y,Z,W$ is monotone near $t=\infty$; from (\ref{form}), the
corresponding solutions $(u,v)$ of system (\ref{sab}) satisfy (\ref{fur}).

We consider the trajectory $\mathcal{T}_{\vec{u}}$ corresponding to $\vec{u}.$
Let us show that the convergence is monotone in all $\mathbb{R}$. Notice that
neither of the components can vanish, since system (\ref{sab}) is of
Kolmogorov type. Near $t=\infty$, $X$ and $Y$ are increasing, and $Z,W$ are
decreasing. Suppose that there exists a greatest value $t_{1}$ such that $X$
has a minimum local at $t_{1}$, hence
\[
X_{tt}(t_{1})=X(t_{1})Z_{t}(t_{1})\geq0,\qquad Z(t_{1})=N-2-X(t_{1}),
\]
thus $Z_{t}(t_{1})\geq0$ . Then there exists $t_{2}\geq t_{1}$ such that
$Z_{t}(t_{2})=0$, and
\[
Z_{tt}(t_{2})=-\delta Z(t_{2})Y_{t}(t_{2})\leq0,\qquad Z(t_{2})=N+a-\delta
Y(t_{2}),
\]
then $Y_{t}(t_{2})\leq0$. There exists $t_{3}\geq t_{2}$ such that
$Y_{t}(t_{3})=0$, and%
\[
Y_{tt}(t_{3})=Y(t_{3})W_{t}(t_{3})\geq0,\qquad Y(t_{3})=N-2-W(t_{3}).
\]
There exists $t_{4}\geq t_{3}$ such that $W_{t}(t_{4})=0$ and $W_{tt}%
(t_{4})=-W(t_{4})X_{t}(t_{4})\leq0$. From the definition of $t_{1},$ this
implies $t_{4}=t_{1},$ and then all the conditions above imply that
$(X,Y,Z,W)(t_{1})=M_{0},$ which is impossible. Hence $X$ stays strictly
monotone, and similarly $Y,Z,W$ also stay strictly monotone. Since $X,Y>0$,
and $Y,Z<0$, then $\mathcal{T}_{\vec{u}}$ is bounded, hence defined on
$\mathbb{R}$ and converges to some fixed point $L=(l_{1},l_{2},l_{3},l_{4})$
of the system as $t\rightarrow-\infty$ and necessarily $l_{1}<X_{0}%
,l_{2}<Y_{0},l_{3}>Z_{0},l_{4}>W_{0}$.\medskip

$\bullet$ Case $N>2$. First we note that along $\mathcal{T}_{\vec{u}}$ we
always have $X,Y>N-2$. Indeed, if at some point $t$ we have $X(t)=N-2$, then
$X_{t}(t)=(N-2)Z(t)<0$, which is contradictory. Hence the possible values for
$L$ are $A_{0}$, or $P_{0}$ when $\mu\geq\frac{N+b}{N-2}$, or $Q_{0}$ when
$\delta\geq\frac{N+a}{N-2},$ since $I_{0}$ is nonadmissible. By hypothesis,
$\gamma_{a,b}>N-2,$ then either $\mu<\frac{N+b}{N-2}$ or $\delta<\frac
{N+a}{N-2}.$ We can assume that $\delta<\frac{N+a}{N-2}$. Then $Q_{0}%
\not \in \mathcal{R}$, then $L=A_{0}$ or $P_{0}$. When $\mu<\frac{N+b}{N-2}$,
then $L=A_{0}$. When $\mu>\frac{N+b}{N-2}$, from Proposition \ref{ex1}(i), we
have $L\neq A_{0}$, thus $L=P_{0}$. In the limit case $\mu=\frac{N+b}{N-2}$,
we find $P_{0}=A_{0}$. From the linearization at $A_{0}$ we have
\[
\lambda_{1}=\lambda_{2}=N-2>0,\;\lambda_{3}=N+a-(N-2)\delta>0,\;\lambda
_{4}=0.
\]
Coming back to the proof of Proposition \ref{ex1}(i), we find that the
convergence of $Z$ and $\tilde{X}=X-(N-2)$ to $0$ are exponential. From the
fourth equation in (\ref{mn}) we see that $W_{t}+W^{2}>0$, hence $-1/W\leq
C|t|$ near $-\infty$. Then, there exists $m>0$ such that
\[
W_{t}=W^{2}(-1-\mu W^{-1}\tilde{X})=W^{2}(-1+O(e^{mt});
\]
integrating over $(t,t_{0})$, $t_{0}<0$, we obtain that $W(t)=t^{-1}%
+O(t^{-2})$. In turn we estimate $Y$; setting $\overline{Y}=\tilde{Y}+W$, then
$\overline{Y}_{t}=(N-2)\overline{Y}+\overline{Y}(\overline{Y}-W)+W(-\mu
\tilde{X}-W),$ and thus
\[
\overline{Y}_{t}=((N-2)+\varepsilon(t))\overline{Y}+O(t^{-2}),
\]
implying $\overline{Y}=O(t^{-2})$ and thus $Y=N-2-t^{-1}+O(t^{-2}).$ Next we
find that $Z_{t}/Z=\lambda_{3}+t^{-1}+O(t^{-2}),$ which yields $\lim
_{t\rightarrow-\infty}e^{-\lambda_{3}t}|t|^{-\delta}|Z|=C>0.$ Finally, by
replacing in (\ref{form}), and deduce the behavior of $u$ and $v$ as claimed:
\[
\lim_{r\rightarrow0}r^{N-2}u=C_{1}>0\quad\text{and}\quad\lim_{r\rightarrow
0}r^{N-2}|\log(r)|^{-1}v=C_{2}>0.
\]

$\bullet$ Case\textbf{\ }$N=2$. Then necessarily $L=O=(0,0,0,0)$. The
eigenvalues of the linearized problem at this point are $0,0,2+a,2+b$. Since
$Z_{t}=Z(2+a-\delta Y-Z)$ and $Y$ and $Z$ tend to $0$ as $t$ tends to
$-\infty$, $Z$ converges exponentially to $0$, and similarly $W$. Since
$X_{t}\leq X^{2}$, it follows that $X\geq C\left\vert t\right\vert ^{-1}$ near
$-\infty$. Then
\[
X_{t}=X^{2}(1+Z/X)=X^{2}(1+O(e^{mt}))
\]
for some $m>0$, hence $X=-1/t+O(t^{-2})$, then the function $t\mapsto
\varphi=u(t)/t$ satisfies $\varphi_{t}/\varphi=O(t^{-2})$, then $\varphi$ has
a finite limit, hence $u(r)/\ln r$ has a finite positive limit, and similarly
for $v.$
\end{proof}


\begin{thebibliography}{99}                                                                                               %


\bibitem {be}\textsc{C. Bandle and M. Essen}, On the solutions of quasilinear
elliptic problems with boundary blow-up, \emph{Sympos. Math.} \textbf{35
}(1994), 93-111.

\bibitem {bm1}\textsc{C. Bandle and M. Marcus}, 'Large' solutions of
semilinear elliptic equations: Existence, uniqueness and asymptotic behavior,
\emph{J. Anal. Math.} \textbf{58 }(1992), 9-24.

\bibitem {bm2}\textsc{C. Bandle and M. Marcus}, On second order effects in the
boundary behavior of large solutions of semilinear elliptic problems,
\emph{Differential and Integral Equations} \textbf{11 }(1998), 23-34.

\bibitem {BGi}\textsc{Bidaut-V\'{e}ron, M., and Giacomini, H.} A new dynamical
approach of Emden-Fowler equations and systems, arXiv:1001.0562v2 [math.AP],
\emph{Adv. Diff. Eq.}, to appear.

\bibitem {BGr}\textsc{M-F. Bidaut-V\'{e}ron and P. Grillot,} Singularities in
elliptic systems with absorption terms, \emph{Ann. Scuola Norm. Sup. Pisa CL.
Sci}\textit{\ }\textbf{28} (1999), 229-271.

\bibitem {cod}\textsc{E. A. Coddington and N. Levinson, } Theory of Ordinary
Differential equations, McGraw-Hill, 1955.

\bibitem {CoDu}\textsc{0. Costin and L. Dupaigne, }Boundary blow-up solutions
in the unit ball: Asymptotics, uniqueness and symmetry, \emph{J. Diff. Equat.
}\textbf{249} (2010), 931-964.

\bibitem {DDGM}\textsc{J. Davila, L. Dupaigne, 0. Goubet and S. Martinez},
Boundary blow-up solutions of cooperative systems, \emph{Ann.
I.H.Poincar\'{e}-AN} \textbf{26} (2009), 1767-1791

\bibitem {delp-l}\textsc{M. Del Pino and R. Letelier}, The infuence of domain
geometry in boundary blow-up elliptic problems, \emph{Nonlinear Anal.}
\textbf{48 }(2002), 897-904.

\bibitem {di-l}\textsc{G. D\'{\i}az and R. Letelier}, Explosive solutions of
quasilinear elliptic equations: Existence and uniqueness, \emph{Nonlinear
Anal.} \textbf{20 }(1993), 97-125.

\bibitem {DLS}\textsc{J.I. D\'{\i}az, M. Lazzo and P.G. Schmidt, }Large radial
solutions of a plolyharmonic equation with superlinear growth,
\emph{Electronic J. Diff. Equat.}, Conference \textbf{16 }(2007)\emph{,} 103-128.

\bibitem {gms}\textsc{J. Garc\'{\i}a-Meli\'{a}n and A. Su\'{a}rez}, Existence
and uniqueness of positive large solutions to some cooperative elliptic
systems, \emph{Advanced Nonlinear Studies} \textbf{3 }(2003), 193-206.

\bibitem {gmr1}\textsc{J. Garc\'{\i}a-Meli\'{a}n and J. D. Rossi}, Boundary
blow-up solutions to elliptic systems of competitive type, \emph{J. Diff.
Equat.} \textbf{206 }(2004), 156-181.

\bibitem {gm2}\textsc{J. Garc\'{\i}a-Meli\'{a}n,} Large solutions for an
elliptic system of quasilinear equations, \emph{J. Differential Equat.,}
\textbf{245 }(2008), no. 12, 3735--3752.

\bibitem {gmls2}\textsc{J. Garc\'{\i}a-Meli\'{a}n, R. Letelier-Albornoz and J.
Sabina de Lis}, The solvability of an elliptic system under a singular
boundary condition, \emph{Proc. Roy. Soc. Edinburgh }\textbf{136 }(2006), 509-546.

\bibitem {lm2}\textsc{A. C. Lazer and P. J. Mckenna}, Asymptotic behavior of
solutions of boundary blow- up problems, \emph{Differential and Integral
Equations} \textbf{7 }(1994), 1001-1019.

\bibitem {LoNi}\textsc{C. Loewner and L. Nirenberg, }Partial differential
equations invariant under conformal or projective transformations,
Contributions to analysis, Acad. Press, New York (1974), 245-272.

\bibitem {lr}\textsc{H. Logemann and E.P. Ryan,} Non-autonomous systems:
asymptotic behavior and weak invariance principles, \emph{Journal of Diff.
Equat.}, \textbf{189 }(2003), 440-460.

\bibitem {marv1}\textsc{M. Marcus and L. V\'{e}ron}, Uniqueness and asymptotic
behavior of solutions with boundary blow-up for a class of nonlinear elliptic
equations, \emph{Ann. Inst. H. Poincar\'{e} Anal. Non Lin\'{e}aire} \textbf{14
}(1997), 237-274.

\bibitem {marv2}\textsc{M. Marcus and L. V\'{e}ron}, Existence and uniqueness
results for large solutions of general nonlinear elliptic equations, \emph{J.
Evol. Equ.} \textbf{3} (2003), 637-652.

\bibitem {MHTL}\textsc{C. Mu, S. Huang, Q. Tuian, L. Liu,} Large solutions for
an elliptic system of competitive type: existence, uniqueness and asymptotic
behavior, \emph{Nonlinear Anal.} \textbf{71} (2009), 4544-4552.

\bibitem {veron}\textsc{L. V\'{e}ron,} Semilinear elliptic equations with
uniform blowup on the boundary, \emph{J. Anal. Math.} \textbf{59 }(1992), 2-250.

\bibitem {veron2}\textsc{L. V\'{e}ron,} Singularities of solutions of second
order quasilinear equations, \emph{Pitman Research Notes in Math. Series,
Lonman, Harlow }\textbf{\ }(1996).

\bibitem {Wa}\textsc{Y. Wang,} Boundary blow-up solutions for a cooperative
system of quasinear equation, preprint.

\bibitem {WuYa}\textsc{M. Wu and Z. Yang,} Existence of boundary blow-up
solutions for a class of quasilinear elliptic systems with critical case,
\emph{Applied Math. Comput.} \textbf{198} (2008), 574-581.

\bibitem {Y1}\textsc{C. Yarur, }Nonexistence of positive singular solutions
for a class of semilinear elliptic systems, \emph{Electronic J. Differential
Equat.} \textbf{8} (1996), 1-22.

\bibitem {Y2}\textsc{C. Yarur, }A priori estimates for positive solutions for
a class of semilinear elliptic systems, \emph{Nonlinear Anal.} \textbf{36}
(1999), 71-90.
\end{thebibliography}
\end{document}